\documentclass[11pt,reqno]{amsart}
\usepackage{amsmath,amsfonts,amssymb,amscd,amsthm,amsbsy,bbm, epsf,calc,  pdfsync}
\usepackage{color}
\usepackage{datetime}

\newcounter{hours}\newcounter{minutes}

\newtheorem{thm}{Theorem}[section]
\newtheorem{question}[thm]{Question}
\newtheorem{lem}[thm]{Lemma}
\newtheorem{cor}[thm]{Corollary}
\newtheorem{prop}[thm]{Proposition}

\newtheorem{DEF}[thm]{Definition}
\theoremstyle{remark}                  
\newtheorem{rem}[thm]{Remark}

\def\D{{\mathcal D}}

\def\L{{\mathcal L}}
\def\M{{\mathcal M}}

\def\S{{\mathbb S}}
\def\T{{\mathcal T}}

\def\real{{\mathbb R}}

\def\Indicator{{\mathbbm{1}}}

\def\ep{\varepsilon}
\def\al{\alpha}
\def\del{\delta}

\def\om{\omega}
\def\Om{\Omega}
\def\gam{\gamma}
\def\Gam{\Gamma} 
\def\lam{\lambda}
\def\Lam{\Lambda}
\def\lap{\Delta}
\def\sig{\sigma}
\def\grad{\nabla}
\def\diam{\textnormal{diam}}

\def\det{\textnormal{det}}
\def\Tr{\textnormal{Tr}}
\def\Id{\textnormal{Id}}

\def\intersect{\bigcap}

\newcommand{\abs}[1]{\left| #1 \right|}

\newcommand{\ds}{\displaystyle}

\newcommand{\Laplace}[1]{(-\Delta)^{#1} }
\newcommand{\norm}[1]{\lVert#1\rVert}
\newcommand{\inner}[2]{#1 \cdot #2}

\newcommand{\BlueComment}[1]{\color{blue}{#1}\color{black}}

\begin{document}

\title{Aleksandrov-Bakelman-Pucci Type Estimates For Integro-Differential Equations}
\author{Nestor Guillen}
\address{Department of Mathematics, University of Texas, 1 University Station C1200, Austin, TX 78712-0257} 
\email{nguillen@math.utexas.edu}

\author{Russell W. Schwab} 
\address{ Department of Mathematical Sciences, Carnegie
    Mellon University, Pittsburgh, PA 15213}
\email{rschwab@andrew.cmu.edu}
\begin{abstract}
In this work we provide an Aleksandrov-Bakelman-Pucci type estimate for a certain class of fully nonlinear elliptic integro-differential equations, the proof of which relies on an appropriate generalization of the convex envelope to a nonlocal, fractional-order setting and on the use of Riesz potentials to interpret second derivatives as fractional order operators. This result applies to a family of equations involving some nondegenerate kernels and as a consequence provides some new regularity results for previously untreated equations.  Furthermore,  this result also gives a new comparison theorem for viscosity solutions of such equations which only depends on the $L^\infty$ and $L^n$ norms of the right hand side, in contrast to previous comparison results which utilize the continuity of the right hand side for their conclusions.  These results appear to be new even for the linear case of the relevant equations.
\end{abstract}

\date{\today\ Arxiv Version} 
\thanks{Nestor Guillen was partially supported by NSF grant DMS-0654267 and Russell Schwab was supported by a NSF Postdoctoral Research Fellowship, grant DMS-0903064.  This work was carried out during multiple visits of Nestor Guillen to The Center for Nonlinear Analysis at Carnegie Mellon University for whose support the authors are also grateful.  The authors would like to thank Takis Souganidis and Luis Silvestre for comments on the preliminary version of this manuscript, and the authors would like to extend a special thanks to the anonymous referee for helpful observations which led to some significant improvements in this work.}
\keywords{Convex Envelope, Integro-Differential Equations, Jump Processes, Levy Processes, Nonlocal Elliptic Equations, Obstacle Problems, Interpolation}
\subjclass[2000]{35J99, 
45J05, 
47G20, 
49L25, 
49N70, 
60J75, 
93E20 
}

\maketitle
\baselineskip=14pt
\pagestyle{headings}		
\markboth{N. Guillen and R. Schwab}{ABP Estimates For Integro-Differential Equations}

\section{Introduction}\label{sec:Intro}
\setcounter{equation}{0}
	
We begin this work with a very basic question to motivate our results.  Suppose that $L$ is a uniformly elliptic operator ($L$ could be a second order or integro-differential operator) and that $u_k$ are appropriate weak solutions (read viscosity solutions) of the equations
\begin{equation*}
\begin{cases}
\ds L(u_k,x)= f_k(x) &\text{ in } B_1\\
u= 0 &\text{ on } \real^n\setminus B_1,
\end{cases}
\end{equation*}
with the additional assumption that $0\leq f_k\leq1$.
Then we pose the following question:

\begin{question}\label{Ques:ABasicQuestion}
Under what conditions will it be true that $\abs{\{x:f_k(x)>0\}}\to0$ as $k\to\infty$ also implies that $\norm{u_k}_{L^\infty}\to0$?
\end{question}

\noindent
(We will assume $L$ is 1-homogeneous, and hence the constant $0$ function is a supersolution, and so always $u_k\leq0$.)

In the case that $L$ is a second order, uniformly elliptic operator,
\begin{equation*}
L(u,x)=a_{ij}(x)u_{x_ix_j}(x),
\end{equation*}
(for $\lam\Id\leq (a_{ij})\leq\Lam\Id$) the answer to the above question is indeed affirmative and is given by the celebrated Aleksandrov-Bakelman-Pucci (ABP) estimate which states that
\begin{equation*}
-\inf_{B_1}\{u_k\} \leq \frac{C(n)}{\lam}\norm{f_k}_{L^n}
\end{equation*}
(where the equation is set in $B_1\subset\real^n$).

The current understanding of this question is dramatically different, and not many results are known when $L$ is a uniformly elliptic integro-differential operator, namely
\begin{equation}\label{eq:IntroLDef}
L(u,x)=\int_{\real^n}(u(x+y)-u(x)-\inner{Du(x)}{y}\Indicator_{\abs{y}\leq1}(y))K(x,y)dy,
\end{equation}
\begin{equation*}
\mbox{where } \lam \leq K(x,y)|y|^{n+\sig}\leq\Lam	\;\forall \;x,y \in \mathbb{R}^n.
\end{equation*}

\noindent
In fact, using current results the only occasions in which something could be said about the possibility of $\norm{u_k}_{L^\infty}\to0$ would be when an explicit Green's function for the operator $L$ in $B_1$ is known to exist.  For many applications and also the possibility of treating nonlinear equations, this is an unsatisfactory answer to Question \ref{Ques:ABasicQuestion}.

In this work, we present a new Aleksandrov-Bakelman-Pucci (ABP for short) type estimate for subsolutions and supersolutions of integro-differential equations for particular kernels $K(x,y)$. Namely, the kernels for which the result applies include (see Remark \ref{rem:Kernels})

\begin{equation*}
K(x,y)= (2-\sig)\frac{y^TA(x)y}{\abs{y}^{n+\sig+2}}\ \ \text{for} \;\;\ \sig\in(0,2),
\end{equation*}

\noindent
where $A(x)$ satisfies a nondegenerate\  ellipticity condition

\begin{equation}\label{eq:DegenerateElliptic}
\Tr(A(x))\geq\lam\ \forall\ x
\end{equation}

\noindent
(and is only measurable in $x$).  
Since these kernels are symmetric in $y$, we may rewrite $L(u,x)$ as
\begin{equation}\label{eq:LinearGUISCH}
L(u,x)=(2-\sig)\int_{\real^n}\del u(x,y) \frac{y^TA(x)y}{\abs{y}^{n+\sig+2}}dy, \;\;\sig\in(0,2)	
\end{equation}

\noindent
where we use the notation for second differences as
\begin{equation*}
\del u(x,y):=u(x+y)+u(x-y)-2u(x).	
\end{equation*}

We want to point out here the much richer structure inherent in elliptic integro-differential equations regarding uniformly and non-uniformly elliptic operators.  A distinction must be made between the properties of the kernels which are used to construct operators such as (\ref{eq:IntroLDef}) and the properties of the operators, $L$, themselves.  Indeed, it is very reasonable to consider $L$ to be uniformly elliptic even if the corresponding $K$ are not pointwise comparable uniformly to the standard kernel, $\abs{y}^{-n-\sig}$ (see \cite{KaMi-2011NondegJumpPreprint} for recent regularity results in this direction).  Such distinctions are by no means well resolved in the current literature, and this should be considered as evidence to the vast amount of work which is still to be done in this direction of integro-differential equations.

Moreover, our result also covers fully nonlinear operators, such as
\begin{equation}\label{eq:Fullynonlinear}
F(u,x)=\sup_\beta\inf_\alpha\left\{L_{\al\beta}(u,x)\right\}
\end{equation}
where each $L_{\alpha\beta}$ is as in (\ref{eq:LinearGUISCH}), with a fixed $\sig$ for all of the $L_{\al\beta}$. It gives new results for the Dirichlet problem:
\begin{equation}\label{eq:PIDEgeneral}
\begin{cases}
\ds F(u,x)= f(x) &\text{ in } D\\
u=g &\text{ on } \real^n\setminus D,
\end{cases}
\end{equation}
where $F$ is given by (\ref{eq:Fullynonlinear}), $g$ is continuous, and $D$ is an open, bounded domain.

The ABP estimate is not only linked to convergence questions as explained above, but it also plays an important role in the regularity theory for solutions of equations related to (\ref{eq:LinearGUISCH})-(\ref{eq:PIDEgeneral}).  In particular, Theorem \ref{thm:Main}, below, implies both new comparison and regularity results for solutions of (\ref{eq:PIDEgeneral}) and gives a new proof of the recent H\"older regularity results of \cite[Section 12]{CaSi-09RegularityIntegroDiff} for the operators which are in the class exemplified by (\ref{eq:LinearGUISCH}).  Furthermore, regularity for equations with operators whose kernels obey the nondegenerate ellitpticity requirement in (\ref{eq:DegenerateElliptic}) have not been previously attained, and so in Section \ref{sec:NonlocalLepsilon} we briefly explain how to obtain H\"older regularity results as a consequence of Theorem \ref{thm:Main}. 

The paradigm of ellipticity taken in \cite[Section 3]{CaSi-09RegularityIntegroDiff} (and \cite[Chapter 2 and 5]{CaCa-95} for second order theory) is that $F$ in (\ref{eq:PIDEgeneral}) is considered elliptic if there are minimal and maximal operators, $M^-$ and $M^+$, such that
\begin{equation*}
M^-(u-v)\leq F(u,x)-F(v,x)\leq M^+(u-v).
\end{equation*}
In this work, we use a more restricted version of these minimal/maximal operators than those in \cite[Section 3]{CaSi-09RegularityIntegroDiff}, which is reflected in the list of kernels covered as (\ref{eq:LinearGUISCH}) being much smaller than, but not exactly contained within  those considered in \cite{CaSi-09RegularityIntegroDiff}.  Our extremal operators are defined as 
\begin{equation}\label{eq:MminusDef}
M^-(u,x) = \inf_{\lam\leq \Tr(A)\ \text{and}\ A\leq\Lam\Id}\left\{(2-\sigma)\int_{\real^n}\del u(x,y) \frac{y^TAy}{\abs{y}^{n+\sig+2}}dy \right\}
\end{equation}
and
\begin{equation}\label{eq:MPlusDef}
M^+(u,x) = \sup_{\lam \leq \Tr(A)\ \text{and}\ A\leq\Lam\Id}\left\{(2-\sigma)\int_{\real^n}\del u(x,y) \frac{y^TAy}{\abs{y}^{n+\sig+2}}dy \right\}.
\end{equation}
Supersolutions (respectively subsolutions) to (\ref{eq:PIDEgeneral}) for these $F$ in the ellipticity family are characterized by the fact that they are automatically supersolutions of a minimal (respectively maximal) equation. 
Accordingly, we assume throughout this note $u$ to be a viscosity supersolution of a minimal equation with variable right hand side:
\begin{equation}\label{eq:PIDEmain}
\begin{cases}
\ds M^-(u,x)\leq f(x) &\text{ in } B_1\\
u\geq 0 &\text{ on } \real^n\setminus B_1,
\end{cases}
\end{equation}
and without loss of generality we assume that $f\geq0$ and that $u\leq0$ in $B_1$.

\begin{rem}
It is very important to remark that although the restriction of $A\leq \Lam\Id$ is necessary for the definitions of $M^-$ and $M^+$, the ABP-type result presented in Theorem \ref{thm:Main} only depends on the value of $\lam$, and not that of $\Lam$.  This is to be expected for such a result, as seen Section \ref{sec:Proof} with the use of Lemma \ref{lem:NonlocalDetMminusOrder}.
\end{rem}

The main contribution of this work is to provide estimates on the infimum of $u$ in terms of measure theoretic quantities of $f$, in particular for (\ref{eq:PIDEmain}) the $L^\infty$ and $L^n$ norms of $f$. The main result of this work is:

\begin{thm}\label{thm:Main}
Assume that $u\in L^\infty(\real^n)\intersect LSC(\real^n)$ solves (\ref{eq:PIDEmain}) in the viscosity sense, that $M^-$ is given in (\ref{eq:MminusDef}), and that $f\in C(\Bar{B_1})$.  Then there exists a constant, $C(n)$, such that 
\begin{equation*}
-\inf_{B_1}\{u\} \leq \frac{C(n)}{\lam}(\norm{f}_{L^\infty(K_u)})^{(2-\sig)/2}(\norm{f}_{L^n(K_u)})^{\sig/2},
\end{equation*}
where $K_u\subset B_1$ is the coincidence set between $u$ and a special envelope of $u$, made precise in Section \ref{sec:Envelope}, (\ref{eq:ContactSet}).
\end{thm}

\noindent
A few immediate remarks related to Theorem \ref{thm:Main} are in order:

\begin{rem}
For the definitions and basic properties of viscosity solutions of (\ref{eq:PIDEgeneral}) and (\ref{eq:PIDEmain}), the reader should consult the works: \cite{BaChIm-08Dirichlet}, \cite{BaIm-07}, and \cite[Sections 1-5]{CaSi-09RegularityIntegroDiff}.  We emphasize the \emph{sign convention} for subsolutions and supersolutions in this work corresponds to that of \cite{CaSi-09RegularityIntegroDiff} and \cite{CaCa-95}.  A more detailed history and presentation pertaining to viscosity solutions of first and second order equations can be found \cite{CrIsLi-92}.
\end{rem}

\begin{rem}\label{rem:Kernels}
We would like to at least make an attempt to give a concise explanation as to the need for the peculiar restriction that $M^-$ is of the form (\ref{eq:MminusDef}) and hence why we can only treat kernels as in (\ref{eq:LinearGUISCH}).  Roughly speaking, our approach starts with the function $u$, which solves a $\sig$-order equation with $\sig<2$, and takes its Riesz potential  to invert the order by an amount $2-\sig$.  Then to the potential, say $P$, of our original $u$, we can apply familiar second order techniques.  In particular, we use the formula for the determinant of $D^2P(x)$ for those $x$ with $D^2P(x)\geq0$: 
\begin{equation*}
\left(\det(D^2P(x))\right)^{1/n} = \frac{1}{n}\inf\{\Tr(AD^2P(x)) : A\geq0\ \text{and}\ \det(A)=1\}.
\end{equation*}
When this computation is translated back to the original $u$, the derivatives on $P$ can actually be transfered to derivatives on the \emph{kernel} used to construct $P$, which results in a $\sig$-order operator containing the term
\begin{equation*}
L_A(u,x)=(2-\sig)\int_{\real^n}\del u(x,y) \frac{y^TA(x)y}{\abs{y}^{n+\sig+2}}dy.
\end{equation*}
All of these steps are implemented carefully in Sections \ref{sec:Envelope} - \ref{sec:DetD2P}.
\end{rem}

\begin{rem}
Due to the slight degeneracy allowed by $\Tr(A)\geq\lam$ in the kernels, $y^TAy\abs{y}^{-n-\sig-2}$, comprising the definition of $M^-$, Theorem \ref{thm:Main} in fact leads to new regularity results.  Specifically for functions, say $w$, satisfying boundedness of the maximal/minimal operators:
\begin{equation*}
M^+(w,x)\geq -C\ \ \text{and}\ \ M^-(w,x)\leq C,
\end{equation*} 
Theorem \ref{thm:Main} provides a key step for the regularity theory, and this will be discussed further in Section \ref{sec:NonlocalLepsilon}.
\end{rem}

The analysis of fractional order integro-differential equations has gained much attention lately (see Section \ref{sec:Background}).  It seems surprising, however, that despite so much interest, such estimates as those in Theorem \ref{thm:Main} have not been obtained thus far.   At least one need for an ABP type estimate for nonlocal equations was noted in \cite[Remark 3.3]{Schw-10Per} regarding the ability to prove homogenization for stationary ergodic families of integro-differential equations via the methods of \cite{CaSoWa-05}.  This stochastic homogenization result will be presented in \cite{Schw-11Stoch}.  Other potential applications will be briefly discussed in Section \ref{sec:Applications}.  In \cite[Section 8]{CaSi-09RegularityIntegroDiff} an ABP type result was proved, but it involves the maximum of $f$ over a \emph{finite collection of cubes}, representing a Riemann Sum approximation to the usual ABP of second order theory.  The result in \cite[Section 8]{CaSi-09RegularityIntegroDiff} was sufficient for the purposes of regularity theory (for very general kernels), nevertheless it has not been sufficient to answer Question \ref{Ques:ABasicQuestion}, which is needed for the stochastic homogenization.

To prove Theorem \ref{thm:Main} we introduce new machinery, specifically a nonlocal and fractional order replacement for the convex envelope and the Monge-Amp\`ere operator of a Riesz potential of $u$, which can be expressed as a nonlinear integro-differential operator in $u$ itself. This auxiliary operator makes up for current lack of a definitive analogue of the Monge-Amp\`ere operator (which has both a divergence and a non-divergence structure), which has been and continues to be a significant obstacle to understanding the more geometric aspects of the regularity theory for 	nonlinear integro-differential equations.

We conclude this introduction with a brief outline of the sections of the paper.  The two new tools mentioned in the previous paragraph, a replacement for the convex envelope and the auxiliary Monge Amp\`ere operator, are explained in Sections \ref{sec:Envelope} and \ref{sec:DetD2P} respectively.  This new $\sig$-order envelope solves a $\sig$-order equation, just as in the second order setting, \cite[Chapter 3]{CaCa-95}, with the convex envelope and a second order equation.  However, in order to gain access to familiar geometric arguments involving the usual convex envelope, we must push the order of the envelope and its equation up to $2$.  This is done by taking the Riesz potential of the envelope, and its properties are presented in Section \ref{sec:Pressure}.  In Section \ref{sec:DetD2P} we develop the Monge-Amp\`ere operator of the potential as a $\sig$-order operator acting on the envelope. In Section \ref{sec:Proof} we provide the final details of Theorem \ref{thm:Main}.  In Sections  \ref{sec:MongeAmpere}, \ref{sec:ImportantLimitSigmaTo2}, \ref{sec:Useful Consequences}, and \ref{sec:NonlocalLepsilon} we present discussions involving respectively:  the auxiliary operator of Monge-Amp\`ere type, limits of these results as $\sig\to2$, other useful theorems related to Theorem \ref{thm:Main} including comparison, and a proof of the $L^\ep$ estimates of elliptic regularity theory which nearly identically follows the classical one of \cite[Lemmas 4.5 and 4.6]{CaCa-95}, from which the Harnack inequality follows easily. 

\section{Background and Main Ideas}\label{sec:Background}
\setcounter{equation}{0}

\subsection{Historical Background}
Analysis of integro-differential equations is by no means a new field.  It is intimately linked, via their infinitesimal generators, with modeling involving L\'evy or L\'evy-Ito processes-- which are much richer and more general than diffusion processes giving rise to second order non-divergence equations.  A key feature intrinsic in modeling with L\'evy-Ito processes and integro-differential equations is that they allow for \emph{long range} interactions of various forms, and that the underlying stochastic processes can have jumps-- in contrast to diffusion processes, which are continuous.  There are \emph{many} recent applications of these processes and their generators, and we list only a characteristic few: for particle systems and their hydrodynamic limit \cite{GiLe-97}, for financial modeling \cite{CoTa-04} and \cite{CoVo-2005IntDiffOptions}, for optimal control (also related to financial modeling)  \cite{Pham-98}, \cite{Sone-86JumpMarkov}, and \cite{Sone-86State-SpaceSIAM}, for image processing \cite{GiOs-2008NLImagePro}.

There has been recently a growing interest in the analysis of integro-differential equations.  The renewed interest seems not only due to the importance in modeling (L\'evy processes have been around for a very long time) but to advances in both the probabilistic and partial differential equation analysis for treatment of equations related to (\ref{eq:PIDEgeneral}) and (\ref{eq:PIDEmain}).  Here we list only a few references  for recent advances, and suggest the interested reader to consult those references contained within, as a complete list would be impossible.  On the analysis/PDE side,  integro-differential equations were brought to the viscosity solutions framework in \cite{Sone-86State-SpaceSIAM}, \cite{Sone-86JumpMarkov}, and later in \cite{Awat-91}, and the comparison theory for these equations was recently improved and re-presented in \cite{BaIm-07}, with the Dirichlet problem being considered in \cite{BaChIm-08Dirichlet}.  H\"older regularity issues were considered in \cite{Silv-2006Holder} for ``non-divergence'' form equations (cf. \cite[Section 3.6]{Silv-2006Holder} for a discussion of ``divergence'' versus ``non-divergence'' in this context), and a series of works making a uniform theory for both H\"older and higher regularity of more general versions of (\ref{eq:PIDEgeneral}) was done in \cite{CaSi-09EvansKrylov}, \cite{CaSi-09RegularityByApproximation}, and \cite{CaSi-09RegularityIntegroDiff}.  Also by PDE methods, parabolic regularity was obtained for divergence form equations in \cite{CaChVa-10}.  Regularity for equations related to (\ref{eq:PIDEgeneral}) was investigated by different, but still PDE, methods in \cite{BaChIm-11Holder}.  Probabilistic analysis was used to investigate many integro-differential equations in both ``divergence'' and ``non-divergence'' forms to obtain various important results including H\"older regularity, Harnack inequalities, and other finer properties in \cite{AbKa-09}, \cite{BaBaChKa-2009DirichletForms}, \cite{BaKa-05Harnak}, \cite{BaKa-05Holder}, \cite{BaLe-2002Harnack}, and \cite{BaLe-2002TransitionProb}.

\subsection{Another Simple Motivation For (\ref{eq:PIDEmain}) and Theorem \ref{thm:Main}}
Let us now setup another very simple question related to that posed in the Introduction. It arises in the study of obstacle problems, and was utilized in \cite{CaSoWa-05} for homogenization (see \cite{Schw-11Stoch} for the same use in a nonlocal setting).  Consider the solution of an elliptic equation,

\begin{equation}
\begin{cases}
\ds F(v,x) = 0 &\text{ in } B_1\\
v= 0 &\text{ on } \real^n\setminus B_1
\end{cases}
\end{equation}
and the solution to the corresponding obstacle problem with the same operator, $u$, solving

\begin{equation}
u = \sup\{\phi: F(\phi,x)\geq0\  \text{in}\ B_1\  \text{and}\ \phi\leq0\ \text{in all of}\  \real^n\}.
\end{equation}
A reasonable question is: what is the difference between $u$ and $v$?  Naturally, since $u$ is a subsolution of the equation for $v$ and they share the same data in $\real^n\setminus B_1$, then at least we can conclude $v\geq u$.  But what about the reverse inequality?  Ellipticity  (see \cite[Definition 3.1]{CaSi-09RegularityIntegroDiff} and \cite[Chapters 2 and 5]{CaCa-95}) tells us that

\begin{equation*}
M^-(u-v,x)\leq F(u,x)-F(v,x).
\end{equation*}
Two of the key properties of the solutions to obstacle problems is that $u$ is actually a \emph{solution} to $F(u,x)=0$ whenever $u(x)\not=0$, and that $u$ inherits the equation from $0$ when it does happen that $u(x)=0$, which is $F(u,x)\leq F(0,x)$ (here we mean the operator applied to the constant, $0$, function).  Therefore, the relevant inequalities are 
\begin{equation*}
M^-(u-v,x)\leq F(u,x)-F(v,x)\leq \Indicator_{\{u=0\}}(x)F(0,x),
\end{equation*}
which by the boundedness of $F(0,x)\leq C$ becomes

\begin{equation}\label{eq:MotivationEq}
\begin{cases}
\ds M^-(u-v,x)\leq C\Indicator_{\{u=0\}}(x) &\text{ in } B_1\\
u-v= 0 &\text{ on } \real^n\setminus B_1.
\end{cases}
\end{equation}

After considering an $f$ which is a continuous approximation from above of $\Indicator_{\{u=0\}}(x)$, one can clearly see the question at hand and the motivation for this work: How can (\ref{eq:MotivationEq}), and hence also (\ref{eq:PIDEmain}),  be treated in a way which only depends upon \textit{measure theoretic} properties of $f$ and not \textit{continuity} properties of $f$?  

This leads us directly back to Question \ref{Ques:ABasicQuestion}.  The natural setting for the homogenization problem will involve a sequence of $u_k$ and $v$ as above with the additional information that $\abs{\{u_k=0\}}\to0$ as $k\to\infty$ (see \cite[Section 3, specifically Lemma 3.5]{Schw-11Stoch}).  The goal in such a situation is to be able to conclude that 
\begin{equation*}
\norm{u_k-v}_{L^\infty}\to0\ \ \text{as}\ \ k\to\infty.
\end{equation*}
\noindent
Before this work was completed, there seemed to be very little results, in fact the authors found none, which could handle a situation in which the right hand side of (\ref{eq:PIDEmain}) would converge in any sense other than in $L^\infty$ (the relevant results in both \cite{BaChIm-08Dirichlet} and \cite{CaSi-09RegularityIntegroDiff} strongly require $f$ to be continuous and converge to $0$ in $L^\infty$ to conclude similar statements).  At least within the restricted class of operators presented here, Theorem \ref{thm:Main} answers this question, which was not possible with previous known results.

\subsection{Theorem \ref{thm:Main} As a Proof of Concept For Broader Results}
As hinted in the introduction, the full class of nonlinear nonlocal elliptic operators is much richer than simply those given in our definition of $M^-$.  In \cite{CaSi-09RegularityIntegroDiff}, there is a version of $M^-$ which covers many more operators  than the one appearing in (\ref{eq:MminusDef}); it is given as
\begin{align}\label{eq:MminusDefCaSi}
M^-_{CS}(v,x) &= \int_{\real^n}(2-\sig)\left(\lam\frac{(\del v(x,y))^+}{\abs{y}^{n+\sig}}-\Lam\frac{(\del v(x,y))^-}{\abs{y}^{n+\sig}}\right)dy\\
&= \inf_{\lam\leq a(x,y)\leq \Lam }\left\{(2-\sigma)\int_{\real^n}\del v(x,y) \frac{a(x,y)}{\abs{y}^{n+\sig+2}}dy \right\}.
\end{align}
Hence the results of \cite{CaSi-09RegularityIntegroDiff} correspond to a much larger family of equations than the one here.  To go even further, as done in the linear case considered in \cite{BaKa-05Holder}, one could use this same definition with measures, $n(x,dy)$, ``comparable'' in a less restrictive sense to $\abs{y}^{-n-\sig}dy$ as opposed to only those measures with a density as above, $n(x,dy)=a(x,y)\abs{y}^{-n-\sig}dy$, for $a$ uniformly bounded from below and above.

A reasonable guess for an ABP type theorem applying to (\ref{eq:PIDEmain}) in these classes above would be 
\begin{equation}\label{eq:ABPwithoutInfinityNorm}
-\inf_{B_1}(u)\leq \frac{C}{\lam}\norm{f}_{L^p},
\end{equation}
for some $p>n$, depending on $\sig$.  At the time \cite{CaSi-09RegularityIntegroDiff} and \cite{Schw-10Per} were completed, it was not known whether or not such a result \emph{should} be true for the general nonlocal ellipticity class or even a restricted class such as the one considered here.  Now Theorem \ref{thm:Main} indicates that at least some form of ABP type result holds for a restricted class of operators.  This gives hope that ABP type estimates, such as those in Theorem \ref{thm:Main}, for the general class of nonlocal elliptic equations may still be true.  Furthermore, it opens the door to answering the question of whether or not (\ref{eq:ABPwithoutInfinityNorm}) is appropriate to expect for (\ref{eq:PIDEmain}).  The moral of the story is Theorem \ref{thm:Main} indicates it is not that ABP type \emph{results} are incompatible with the intrinsic properties of (\ref{eq:PIDEmain}), but more importantly that the \emph{existing machinery} is incompatible with (\ref{eq:PIDEmain}).  This illuminates one of the main difficulties in analysis of nonlocal equations: there is no known general framework to take the place of the very important Monge-Amp\`ere operator and the gradient mapping of convex functions in the second order theory.

\subsection{Failure of The Convex Envelope}
For the convex envelope, say $\Gamma$, the information in the proof of the original second order ABP estimate is naturally encoded in the set $\{x:\det(D^2\Gamma(x))>0\}$.  This matches very well with second order equations because $\det(D^2\Gamma)=0$ whenever $\Gamma$ and $u$ do not coincide, and on the set where they do coincide, $\Gamma$ inherits the supersolution property of $u$ simply by comparison.  In the nonlocal setting, the function $w=\abs{x}^{\al}-1$ will solve (\ref{eq:PIDEmain}) for $\al>\sig$ with a right hand side, $f$, which will still be bounded and continuous.  However for $\sig<\al<1$, the \textit{convex} envelope of $w$ only coincides with $w$ at one point, $x=0$.  Hence there would be a contradiction with the usual ABP estimate which would read:
\begin{equation*}
-\inf_{B_1}\{w\} \leq C\norm{f}_{L^n(\{w=\Gamma\})}=0.
\end{equation*}
The problem here is exactly that such fractional order equations allow for much more drastic bending of supersolutions than is possible in a second order setting.  Therefore, the convex envelope is not well suited for a measure theoretic estimate, such as Theorem \ref{thm:Main}.  For purposes of studying the regularity of (\ref{eq:PIDEgeneral}) with even more general operators than the $F$ appearing in (\ref{eq:Fullynonlinear}), the convex envelope was sufficient and strongly used in \cite[Section 8]{CaSi-09RegularityIntegroDiff}.  In contrast to \cite{CaSi-09RegularityIntegroDiff}, in this work we must construct a different envelope which will be better suited to handle such a function as $w$, above.  This will be the content of Section \ref{sec:Envelope}

\subsection{Main Ideas and Sketch of The Proof}\label{sec:Sketch}
Here we give a brief sketch of how Theorem \ref{thm:Main} is proved.  The main ideas are the same as in the second order theory (cf. \cite[Chapter 3]{CaCa-95} or \cite[Chapter 9]{GiTr-98}), but the machinery and implementation are a bit more involved.  

Everything starts with an appropriate envelope of $u$ \emph{from below}, which we will denote as $\Gamma$. It must be such that
\begin{equation}\label{eq:SketchInfGamma}
\inf\{u\}=\inf\{\Gamma\}.
\end{equation}

The first key feature is the existence of an operator, which we denote as $\D_\sig$, such that 
\begin{equation}\label{eq:SketchIntegralDsigma}
\left(-\inf{\Gamma}\right)^{p} \leq C\int_{B_3} \left(\D_\sig(\Gamma,x)\right)^{p}dx
\end{equation}
for some $p$ possibly depending on $n$ and $\sig$.  At this point for the second order theory, we would have $\Gamma$ as the convex envelope of $u$ and $\D_\sig$ would be $\det(D^2\Gamma)^{\BlueComment{1/n}}$, in which case the previous inequality is a consequence of \emph{the geometry} of convex functions (what is known as ``Aleksandrov's estimate'').  The second key feature is that the operator $\D_\sig$ must satisfy \emph{for a special class of} $\Gamma$,
\begin{equation}\label{eq:SketchDsigmaMminus}
\D_\sig(\Gamma,x)\leq C M^-(\Gamma,x).
\end{equation}
Finally, the third key feature is that
\begin{equation}\label{eq:SketchDsigmaOffContactSet}
\D_\sig(\Gamma,x)\leq0\  \text{whenever}\  \Gamma(x)\not=u(x).
\end{equation}
This is essential so that all the contribution of $\D_\sig(\Gamma)$ can be ignored except for the contact set between $u$ and $\Gamma$.  This way, information about $M^-(u,x)\leq f(x)$ in the viscosity sense can be carried over to $\Gamma$ via comparison, and the other values of $\D_\sig(\Gamma)$ will not pollute the integral in the estimate (\ref{eq:SketchIntegralDsigma}).  Essentially $\Gamma$ acts as a test function on $u$, and the defining feature of viscosity solutions is that at those points where $\Gamma$ touches $u$ from below, 
\begin{equation}\label{eq:SketchMminusGammafx}
M^-(\Gamma,x)\leq f(x).
\end{equation}
If all of (\ref{eq:SketchInfGamma}), (\ref{eq:SketchIntegralDsigma}), (\ref{eq:SketchDsigmaMminus}), (\ref{eq:SketchDsigmaOffContactSet}), and (\ref{eq:SketchMminusGammafx}) can be satisfied (which is very nontrivial), then Theorem \ref{thm:Main} can be proved.

It turns out to be quite difficult to simultaneously achieve all three of the key features, (\ref{eq:SketchIntegralDsigma}), (\ref{eq:SketchDsigmaMminus}), and (\ref{eq:SketchDsigmaOffContactSet}).  This delicate balance is what leads to Theorem \ref{thm:Main} only being proved for a restricted class of operators, instead of the much more general class of \cite{CaSi-09RegularityIntegroDiff}. 

\subsection{Notation}\label{subsec:Notation}
We list here some notation which will be used throughout this work.
\begin{enumerate}
\item The second difference operator: $\del v(x,y):= v(x+y)+v(x-y)-2v(x)$
\item The $n-1$ dimensional sphere $S^{n-1}=\partial B_1\subset \real^n$ and its surface area $\om_n$
\item The complement of a set, $A^c=\real^n\setminus A$
\item The following universal constant will appear often (when $n\geq 2$), \[A(n,\alpha)= \pi^{\alpha-\frac{n}{2}}\frac{\Gamma(\frac{n-\alpha}{2})}{\Gamma(\frac{\alpha}{2})}\]
It can be showed that $A(n,\alpha) \sim \alpha$ as $\alpha \to 0^+$ (see \cite[Chapter I, page 44]{Land-72}).
\item The Riesz Potential of order $\al$, $\ds K_\al(y):= A(n,\al) \abs{y}^{-n+\al}$
\item The Fractional Laplacian (when $n\geq 2$)
\begin{equation*}
-\Laplace{\sig/2}v(x) = \tfrac{\sig (n+\sig-2)}{2}A(n,2-\sig)\int_{\real^n}\del v(x,y)\abs{y}^{-n-\sig}dy.
\end{equation*}

\item The \emph{one dimensional} Fractional Laplacian in the direction $\tau \in\S^{n-1}$, (for some universal constant $(A(1,2-\sig)$ which we will not need to specify here),
\begin{equation*}
-(-\Delta)^{\sigma/2}_\tau v(x) :=\tfrac{\sig (1+\sig-2)}{2}A(1,2-\sig) \int_{\real}(\delta v(x,s\tau))\abs{s}^{-1-\sigma}ds
\end{equation*}
\item The inf-convolution of a function $v$, $v_\ep(x):=\inf_{y}\{u(y)+(2\ep)^{-1}\abs{x-y}^2\}$ (See \cite[Equations (14), (15)]{JeLiSo-88UniquenessSecondOrder}, \cite{LaLi-1986RegularizationHilbertSpace}, \cite[Appendix]{CrIsLi-92})
\item v is $C^{1,1}$ from above at $x$  (respectively from below) \cite[Definition 2.1]{CaSi-09RegularityIntegroDiff} if there exists a radius $r$, a vector $p$, and a constant $M$ such that for all $\abs{y}\leq r$ 
\begin{equation*}
v(x+y)-v(x)-\inner{p}{x}\leq M\abs{y}^2
\end{equation*}
(respectively $v(x+y)-v(x)-\inner{p}{x}\geq M\abs{y}^2$)

\item Contact Set between $u$ and its envelope, $\Gamma$ (defined in Definition \ref{def:GammaU}), 
\begin{equation*}
K_u:=\{x:u(x)=\Gamma(x)\}
\end{equation*}
\item The convex envelope in $B_R$, $v^{CE}$, for a function $v\geq0$ in $\real^n\setminus B_R$ is 
\begin{equation*}
v^{CE}(x)=\sup\{l(x):l\ \text{affine and}\ l\leq v\ \text{in}\ B_R\}
\end{equation*}
 
\item The extremal operators and ellipticity constants for the family governed by $L$ of (\ref{eq:LinearGUISCH}) are $M^-$ and $M^+$ defined in (\ref{eq:MminusDef}) and (\ref{eq:MPlusDef}).

\end{enumerate}

\section{The Fractional Envelope}\label{sec:Envelope}
\setcounter{equation}{0}

This section is dedicated to constructing a new envelope (as a replacement for the convex envelope)  for the supersolution, $u$, of (\ref{eq:PIDEmain}) which will be essential to proving Theorem \ref{thm:Main}.  The main idea is to imitate the most important features of the convex envelope as they pertain to the second order theory, cf. \cite[Chapter 3]{CaCa-95} and \cite[Proposition 2.12 and Appendix A]{CaCrKoSw-96} for the fully nonlinear version and \cite[Chapter 9, Section 1]{GiTr-98} for the linear version.  The goal is for the new fractional order envelope to be in a class of functions for which $M^-$ is comparable to a nonlocal version of the Monge-Amp\`ere operator, and also to cause this nonlocal Monge-Amp\`ere operator to vanish whenever $u$ and its envelope do not touch.  For the sake of explanation, let $\Gamma$ be the convex envelope of $u$ and $\M^-$ to be the second order minimal Pucci operator (see \cite[Chapter 2, Section 2]{CaCa-95}).  For $\Gamma$, the two requirements just mentioned above correspond to the two facts implied by convexity and the envelope property: 
\begin{align*}
&\prod_{k=1}^ne_k \leq \left(\frac{1}{n}\sum_{k=1}^n e_k\right)^n\leq \left(\frac{1}{\lam}\M^-(\Gamma)\right)^n\ \ \text{for the eigenvalues,}\  e_k,\ \text{of}\ D^2\Gamma\\
\intertext{and}
&\det(D^2\Gamma(x))=0\ \ \text{if}\ u(x)\not=\Gamma(x).
\end{align*}

\noindent
The appropriate analogs to our new envelope appear subsequently as Lemmas \ref{lem:NonlocalDetMminusOrder} and \ref{lem:NonlocalDetContactSet}.\bigskip

In order to define the new envelope, we introduce a matrix-valued integro-differential operator, which will play the (auxiliary) role of an ``integro-differential Hessian''.

\begin{DEF}
Let $n\geq 2$, $\sig\in(0,2)$, and let $v$ be a bounded function such that \[\ds \int_{\real^n}\frac{\abs{\del v(x,y)}}{\abs{y}^{n+\sig}}<\infty,\] 
then we define	
	\label{def:FracHessU}
	\begin{equation}\label{eq:FracHessU}
	h_\sig(v,x) :=  \tfrac{(n+\sig-2)(n+\sig)}{2}A(n,2-\sig) \int_{\mathbb{R}^n}\frac{y \otimes y }{|y|^{n+\sig+2}}\del v(x,y)dy
	\end{equation}
	
\end{DEF}

\noindent
The use of this matrix is dictated by our approach based on Riesz potentials in Section \ref{sec:Pressure}, and it will become more clear in Section \ref{sec:DetD2P}. By this we mean the following, if $v$ is smooth enough and $K_{2-\sig}$ is the Riesz Potential (see section \ref{subsec:Notation}) then a consequence of Lemma \ref{lem:SingularIntegralPij} is

$$D^2 ( v * K_{2-\sig}) = h_\sig(v,x) +\tfrac{(-\Delta)^{\sig/2}v(x)}{n+\sig}\; \Id$$

\begin{rem} It is immediate that if $u$ touches $v$ from above at $x$ and they are both smooth enough then we have the matrix inequality $h_\sig(u,x)\geq h_\sig (v,x)$.
\end{rem}

\begin{rem}\label{rem:EsigSigTo2}
As $\sig \to 2^-$ (and thus $\alpha \to 0$), the above identity gives us
$$h_\sig(v,x) \to D^2v(x) + \tfrac{\Delta v(x)}{n+2} \Id$$.
	
\end{rem}

\begin{rem} The fully nonlinear integro-differential operators that our methods handle are exactly those that can be written in the form
$$I(u,x)=\inf_a\sup_b\{\Tr(A^{ab}(x)h_\sig(u,x)\ :\ \Tr(A^{ab}(x))\geq\lam\ \text{and}\ A^{ab}(x)\leq \Lam\Id\},$$ 	
where $a$ and $b$ can be taken over arbitrary index sets.
\end{rem}

Moving forward, the \textit{``$\sigma$-order envelope''}  of $u$ is defined as the solution of an obstacle problem which is a generalization of the obstacle problem satisfied by the convex envelope.  
First we define an operator which will be a fractional order replacement of the first eigenvalue of the Hessian operator used in the construction of the convex envelope for second order equations.

\begin{DEF}\label{def:Esigma}
Let $\lam_1(B)=\min\{e: e\ \text{is an eigenvalue of}\ B\}$ be the first eigenvalue operator of a matrix.  $E_\sig$ is defined as
\begin{equation}\label{eq:Esigma}
E_\sig(v,x) := \lambda_1 \{ h_\sig(v,x) \} = \inf \limits_{\tau \in S^{n-1}}\{ \inner{(h_\sig(v,x)\; \tau)}{\tau}\}.	
\end{equation}
\end{DEF}
\noindent
Now, using $E_\sig$, we can define our ``\emph{$\sig$-order envelope}''.

\begin{DEF}\label{def:GammaU}
\begin{equation}\label{eq:SigmaEnvelope}
\Gamma^\sigma_u (x) = \sup\left\{v(x) :  E_\sigma( v)\geq 0\  \text{in}\  B_3,\ \text{and}\ v\leq u\Indicator_{B_1}\ \text{in}\ \real^n \right\},
\end{equation}
and the contact set between $\Gamma^\sig_u$ and $u$ will be denoted as
\begin{equation}\label{eq:ContactSet}
K_u:= \{x:\ \Gamma^\sig_u(x)=u(x)\}.
\end{equation}
\end{DEF}

\noindent

\begin{rem}
The usual convex envelope used in the second order theory has an analogous structure as the solution to an obstacle problem.  The interested reader should consult \cite{Ober-07} for an exposition.
\end{rem}

\noindent
By remark \ref{rem:EsigSigTo2}, this new operator, $E_\sig$, does not recover the smallest eigenvalue of the Hessian of $u$ in the limit $\sig \to 2$, instead (see Proposition~\ref{prop:UniformConvergenceEnvelopes})
\begin{equation*}
\lim_{\sig\to2}E_{\sig}(u,x) = \lambda_1(D^2u(x)) +\tfrac{1}{n+2}\Delta u(x).
\end{equation*}
In particular, as $\sigma \to 2^-$ the envelope $\Gamma^\sig_u$ converges to the solution of the upper obstacle problem for the operator above with $u$ as the upper obstacle. The solution to this problem lies above the convex envelope of $u$ but it does not agree with it.

\begin{rem}\label{rem:EsigStar}
All of the results of this section hold (with small modifications) if instead we use
$$E_\sig^*(u,x) := \inf \limits_{\tau \in S^{n-1}} \left \{-(-\Delta)^\sig_\tau u(x) \right \}$$

\noindent Here $(-\Delta)^\sig_\tau u$ is the one dimensional fractional Laplacian of $u$ defined in Section \ref{subsec:Notation}. This is a much more drastic notion of envelope (such an envelope would be below the one defined using $E_\sig$, and touch $u$ on a much smaller domain, which can be problematic). In principle, it is perfectly tailored to handle much more general kernels: 
$$K(y)=\frac{a(y)}{|y|^{n+\sig}} \;\mbox{ where }\; a(r y) = a(y) \;\forall \;r \in \real.$$

\noindent However, it is not yet clear how one can go about proving $L^\infty$ bounds for this envelope in terms of a ``convenient'' integral quantity (i.e. one that can be controlled by and integral of  $ M^-(\Gamma)$). In contrast, the envelope we use admits integral bounds granted by its compatibility with the Riesz potential. Whether this argument work for the more ``drastic'' envelope is not clear and perhaps an entirely different approach is need,  this question will be addressed in future work.

\end{rem}	

\begin{rem}
Some nonlinear one directional operators related to this alternative operator $E^*_\sig$ have also been considered in \cite{BjCaFi-2010TugOfWar}. In particular, the one dimensional Fractional Laplacian in the direction of $\nabla u$ gives rise to a natural integro-differential analogue of the Infinity Laplacian. 
\end{rem}

For the remainder of this section, some properties of $\Gamma^\sig_u$ will be collected for later use.  Also, we will dispense with the notation $\Gamma^\sig_u$ and instead simply use $\Gamma$ except in special cases which will be appropriately noted.

\begin{lem}[Compact Support of $\Gamma$]\label{lem:GammaCompactSupport}
$\Gamma=0$ in $\real^n\setminus B_3$.
\end{lem}
\begin{proof}[Proof of Lemma \ref{lem:GammaCompactSupport}]

This lemma is an immediate consequence of the equation satisfied by $\Gamma$, which is treated as (\ref{eq:AppendixObstacleViscosityEq}) in the Appendix.

\end{proof}

As a solution to an obstacle problem, $\Gamma$ attains regularity from two sources. A one sided regularity of $u$ from above is transferred to $\Gamma$ via the obstacle, and further regularity of $\Gamma$ from below is attained via the structure of the operator, $E_\sig$.  We record the result as it pertains to $\Gamma$ in Proposition \ref{lem:GammaRegularityByAbove}, but the result holds in more generality and is of independent interest-- we have chosen to present these related results in the Appendix, Section \ref{sec:Appendix}.

\begin{prop}[$\Gamma$ Regularizes From Below]\label{lem:GammaRegularityByAbove}
Let $\Om\subset\subset B_3$. If $u$ satisfies respectively 
\begin{itemize}
\item[(i)] $-\Laplace{\sig/2}u\leq C$ classically in $B_1$,
\item[(ii)] $u$ is $C^{1,1}$ from above with a bound, $C$
\end{itemize}
then there exists a $C_1$ depending on $\Om$, $\sig$, $n$, $\lam$, $\Lam$ $\norm{u}_{L^\infty}$ and $C$ such that for a.e. $x\in\Om$, respectively
\begin{itemize}
\item[(i)] $$0\leq h_\sig(\Gamma,x)\leq C_1\Id\ \ \text{and}\ \ \norm{\Gamma}_{H^\sig(\Om)}\leq C_1,$$
\item[(ii)]
\[
\int_{\real^n}\frac{\abs{\del \Gamma(x,y)}}{\abs{y}^{n+\sig}}dy\leq C_1.
\]
\end{itemize}

\end{prop}

\begin{rem}
In particular, Proposition \ref{lem:GammaRegularityByAbove} ensures that for almost every $x$, the $\sig$ order operators can be evaluated classically on $\Gamma$.	

\end{rem}
\begin{rem}
Proposition \ref{lem:GammaRegularityByAbove} would in fact be easier and stronger if the obstacle problem for $\Gamma$ were posed in all of $\real^n$ instead of $B_3$.  In that case, $\Gamma$ would be $C^{1,1}$ from above whenever $u$ was $C^{1,1}$ form above at well, and would follow from a straightforward adaptation of \cite[Proposition 3.10]{Silv-07} to the case of a concave nonlinear operator.
\end{rem}

\begin{proof}[Proof of Proposition \ref{lem:GammaRegularityByAbove}]
The proof is an almost immediate consequence of Proposition \ref{prop:BrezisKinderNonlocal}.  We simply note that the fact that $E_\sig(\Gamma,x)\geq0$ in $B_3$ implies that $h_\sig(\Gamma,x)\geq0$ in $B_3$.  Hence $h_\sig(\Gamma)$ has all non-negative eigenvalues and so $\Tr(h_\sig(\Gamma))$ bounds all of them individually.  Hence because $$\Tr(h_\sig(\Gamma))=-\Laplace{\sig/2}\Gamma,$$ the result (i) follows immediately from Proposition \ref{prop:BrezisKinderNonlocal}.  Result (ii) follows immediately from Corollary \ref{cor:AppendixAllOperators}.
\end{proof}

The following fact is a straightforward consequence of the definitions of $M^-$ and $E_\sig$, we state it as a lemma without proof.

\begin{lem}[Ordering of $E_\sig$ and $M^-$]\label{lem:EsigmaMminusOrder}
For any $v$ which satisfies 
$\ds \int_{\real^n}\frac{\abs{\del v(x,y)}}{\abs{y}^{n+\sig}}<\infty,$
\begin{equation}\label{eq:EsigmaLEQMminus}
E_\sig(v,x) \geq 0 \; \Rightarrow\;\; E_\sig(v,x)\leq \frac{C(n)}{\lam}M^-(v,x).
\end{equation}
\end{lem}

Exactly as with Lemma \ref{lem:EsigmaMminusOrder}, the following is a straightforward consequence of the definitions, and so we again omit the proof.
\begin{lem}\label{lem:GammaSignOnOperators}
If $u$ satisfies classically $-\Laplace{\sig/2}u\leq C$ in $B_1$, then for almost every $x\in B_3$, the operators on $\Gamma$ are classically defined and satisfy for all $A\geq0$:
\begin{equation*}
\int_{\real^n}\del \Gamma(x,y) \frac{y^TAy}{\abs{y}^{n+\sig+2}}dy\geq0.
\end{equation*}
\end{lem}

To wrap up this section, we prove one last property of the envelope, $\Gamma$.  This is the stability of the contact sets with respect to an increasing family of obstacle functions, $u_\ep$.

\begin{lem}\label{lem:ContactSetsLimsup}
If $u_\ep\nearrow u$ as $\ep\searrow0$, then $\limsup K_\ep\subset K$, where $K_\ep=\{x:u_\ep(x)=\Gamma_\ep(x)\}$ and $K=\{x:u(x)=\Gamma(x)\}$ and $\Gamma_\ep$, $\Gamma$ are as in Definition \ref{def:GammaU}. 
\end{lem}

\begin{proof}[Proof of Lemma \ref{lem:ContactSetsLimsup}]
The increasing property of $u_\ep$ implies for $\ep_2<\ep_1$ that $\Gamma_{\ep_1}$ is an admissible subsolution below $u_{\ep_2}$, and hence because $\Gamma_{\ep_2}$ is the largest such subsolution, $\Gamma_{\ep_2}\geq \Gamma_{\ep_1}$.  Thus $\Gamma_{\ep}$ is also increasing, $\Gamma_\ep\nearrow\Gamma$.  Now suppose $x$ is a point in the set $\limsup_{\ep\to0} K_\ep$.  We can extract a subsequence indexed by $\ep_k$ such that $u_{\ep_k}(x)=\Gamma_{\ep_k}(x)$.

It must be shown that $\Gamma(x)=u(x)$.  We already know that $\Gamma(x)\leq u(x)$ by definition of $\Gamma$, so therefore we will show $\Gamma(x)-u(x)\geq0$:
\begin{align*}
0&\geq \Gamma(x)-u(x)\\
&\geq \Gamma(x) - \Gamma_{\ep_k}(x) +\Gamma_{\ep_k}(x)-u_{\ep_k}(x) + u_{\ep_k}(x)-u(x)\\
&\geq 0 + 0 + u_{\ep_k}(x)-u(x).
\end{align*}
Letting $k\to\infty$ ($\ep_k\to0$) yields the result.
\end{proof}

\section{The Fractional-Order Potential}\label{sec:Pressure}
\setcounter{equation}{0}

In the previous section we defined and presented some properties of an important envelope which will be used in the proof of Theorem \ref{thm:Main}.  However, there are still many difficulties to be overcome in proving an ABP type estimate for (\ref{eq:PIDEmain}), specifically one must relate the infimum of $\Gamma$ to an integral of a quantity which is comparable to $M^-(\Gamma)$.  In the second order setting (using a \emph{convex} $\Gamma$), this is resolved by the very special geometry of the convex envelope which gives
\begin{equation*}
-\inf\{\Gamma\}\leq C(n)\int_{B_3}\det(D^2\Gamma(x))dx.
\end{equation*}
A key feature in the proof is that we can transfer this argument to the Riesz potential of $\Gamma$, which will solve a second order equation.  Moreover, the usual argument will result in looking at the gradient measure, $\abs{\grad P(B_3)}$, which is a nonlocal quantity of $\Gamma$.  Inspiration for this choice came from the use of the Riesz potential in the setting of a nonlocal Porous Medium Equation treated in \cite{CaVa-10PorousFracPressure}. This method is also reminiscent of the regularity theory for the obstacle problem for the fractional Laplacian \cite{Silv-07}, which starts with the observation that the fractional Laplacian of the solution (of order $\sig$) solved an equation of order $2-\sig$.

As mentioned, we take our fractional-order potential to be the Riesz potential of the envelope (cf. \cite[Chapter 5, Section 1]{Stei-71} or \cite[Ch. 1, Section 1]{Land-72} ), which formally says

\begin{equation*}
P=(-\lap)^{-(2-\sig)/2} \Gamma.
\end{equation*}
Precisely we mean for this to be the convolution with the Riesz kernel, $K_\al:= A(n,\al) \abs{\cdot}^{-n+\al}$,
\begin{equation}\label{eq:Pressure}
P=\Gamma * K_\al,
\end{equation}
for $\al=(2-\sig)$, and the constant $A(n,\al)$ which is listed explicitly in section \ref{subsec:Notation}.  Hence 
\begin{equation*}
(-\lap) P=(-\lap)^{\sig/2}\Gamma\ \text{and}\ \Laplace{(2-\sig)/2}P=\Gamma.
\end{equation*}

The whole motivation of using the potential, $P$, of order $(2-\sig)/2$ is to make sure that the usual 2nd order Monge-Amp\`ere operator, $\det(D^2P)$, becomes a $\sig$-order operator when viewed as an operation on $\Gamma$.  This way, there is hope that $\det(D^2P)$ will be comparable to $M^-\Gamma$, and familiar techniques from the proof of the second order ABP estimate can be used (see \cite[Lemma 3.4]{CaCa-95} and/or \cite[Lemma 9.2]{GiTr-98}).  Indeed, the form of $\det(D^2P)$ is investigated in Section \ref{sec:DetD2P} and will be shown to be a $\sig$-order operator comparable to $M^-\Gamma$ (this new operator will not be an elliptic $\sig$-order operator in general, but we will be working with a convenient special case).  However, some additional difficulties are introduced by using $P$ instead of $\Gamma$ itself.  It will be necessary to compare the infimum of $P$ to the infimum of $\Gamma$ in a uniform fashion which does not depend on the continuity of $\Gamma$ or $u$.  This is not to be expected simply from the Riesz potential itself, but must be deduced from the equation satisfied by $\Gamma$.  The relationship between $P$ and $\Gamma$ is the main result of this section, which is presented in Proposition \ref{prop:PointToMeasureSIGMA}.  

To begin the results of this section, it will be helpful to remark upon the regularity of $P$ pertaining to that of $\Gamma$.  These results are presented in \cite{Silv-07}, and so we do not provide a proof.  They all follow from either of the two equations involving $P$:
\begin{equation*}
\Laplace{(2-\sig)/2}P=\Gamma
\end{equation*}
or
\begin{equation*}
(-\lap) P=(-\lap)^{\sig/2}\Gamma.
\end{equation*}

\begin{lem}[Proposition 2.9 of \cite{Silv-07}]\label{lem:PEstimates}
The regularity on $P$ inherited from $\Gamma$ is
\begin{itemize}
\item[i)] whenever $(2-\sig)\leq 1$, $P\in C^{0,\al}$ for any $\al<(2-\sig)$ and
\begin{equation*}
\norm{P}_{C^{0,\al}}\leq C(\norm{P}_\infty+\norm{\Gamma}_\infty),
\end{equation*}

\item[ii)] whenever $(2-\sig)>1$, $P\in C^{1,\al}$ for any $\al<(2-\sig)-1$ and
\begin{equation*}
\norm{P}_{C^{1,\al}}\leq C(\norm{P}_\infty+\norm{\Gamma}_\infty),
\end{equation*}
and
\item[iii)] $P\in C^{2,\gam}$ whenever $\Gamma\in C^{1,1}(\real^n)$  (only for $\sig<2$).
\end{itemize}
\end{lem}

There are two more very useful consequences which follow simply from the equation satisfied by $P$.  They will play a significant role later in both Proposition \ref{prop:PointToMeasureSIGMA} and the proof of Theorem \ref{thm:Main}.

\begin{lem}\label{lem:PCantTouchPlane}
If $x_0\in B_3^c$ and $l$ is a linear function, then $P-l$ cannot attain a local minimum at $x_0$.  
\end{lem}

\begin{proof}[Proof of Lemma \ref{lem:PCantTouchPlane}]
As noted above, $P$ globally solves the equation 
\begin{equation*}
\Delta P = -\Laplace{\sig/2}\Gamma.
\end{equation*}
Furthermore, due to the fact that $\Gamma$ satisfies the Dirichlet problem, (\ref{eq:AppendixObstacleViscosityEq}), we know $\Gamma(x_0)=0$ and without loss of generality $\Gamma<0$ in $B_3$, we also have that 
\begin{equation*}
\Delta P(x) = -\Laplace{\sig/2}\Gamma(x) < 0.
\end{equation*}
Thus, by comparison for $\Delta P$ and $\Delta l$, it is impossible for $P-l$ to have a local minimum at $x_0$.
\end{proof}

\begin{cor}\label{cor:InfPB3}
If $x_0$ satisfies $P(x_0)=\inf\limits_{\real^n}\{P\}$, the $x_0\in \overline{B_3}$.
\end{cor}

\begin{cor}\label{cor:D2PNonNegSet}
If $A\subset\real^n$ is the set where $P$ has a second order Taylor Expansion, then
\[
\{x\in A\ :\ D^2P(x)\geq0\} \subset B_3(0).
\]
\end{cor}

Now we proceed with the very interesting relationship between $\inf\{P\}$ and $\inf\{\Gamma\}$ which is imposed by the equation governing $\Gamma$ and will be developed as Lemma \ref{lem:GoodRing} and Proposition and \ref{prop:PointToMeasureSIGMA}.  Lemma \ref{lem:GoodRing} is an example of the type of information that is available when analyzing nonlocal equations which is not available in the local setting.  Basically this next lemma says that the nonlocal equation prevents, in a uniform fashion, $\Gamma$ (and also $u$) from growing away from its minimum faster that $\abs{y}^{\sig}$ in a large set.

\begin{lem}\label{lem:GoodRing}
Let $0<\sig<2$ and $x_0$ be such that
$\Gamma(x_0)=\inf_{\real^n}\{\Gamma\}$.  Define the set
\begin{equation}\label{eq:GoodSetA}
A_{x_0}:= \{y: \Gamma(x_0+y)-\Gamma(x_0)\leq f(x_0)\abs{y}^{\sig}\}
\end{equation}
and rings with radius $r_k=\rho_02^{-k}$, $R_k=B_{r_k}(x_0)\setminus
B_{r_{k+1}}$ (where $\rho_0$ will be chosen later for the proof of
Proposition \ref{prop:PointToMeasureSIGMA}).  If the indices, $k^*$,
indicate bad rings for which
\begin{equation}\label{eq:BadRing}
\abs{A^c_{x_0}\intersect R_{k^*}}\geq \frac{1}{2}\abs{R_{k^*}},
\end{equation}
then the number of such indices satisfies
\begin{equation*}
(\# k^*)\leq \frac{2}{(2-\sig)\lam}C_n
\end{equation*}

\end{lem}

\begin{proof}[Proof of Lemma \ref{lem:GoodRing}]
Because $x_0$ is a location of a global minimum of $\Gamma$ which is
also a contact point with $u$, we know that $M^-(\Gamma,x_0)\leq
f(x_0)$ by comparison with $u$.  Furthermore, there are two things which result from
$\del\Gamma(x_0,y)\geq0$: one, all of the operators appearing in the
definition of $M^-$ (equation (\ref{eq:MminusDef})) are bounded below
by the one corresponding to a matrix $(\lam/n) e \otimes e$ (for some unit vector $e$), and two, we can neglect as
much of the set of integration as we like in the evaluation of
$M^-(\Gamma,x_0)$. Without loss of generality, we may assume $e$ points in the positive direction of the $x_1$-axis.

Therefore, evaluating (\ref{eq:PIDEmain}) at $x_0$ and estimating as
suggested above, we have

\begin{align*}
&f(x_0) \geq M^-(\Gamma,x_0) \\
&\geq \lam\int_{\real^n}
(2-\sig)\del\Gamma(x_0,y)y_1^2\abs{y}^{-n-\sig+2}dy\\
&= \lam\int_{\real^n}
(2-\sig)(\Gamma(x_0+y)-\Gamma(x_0))y_1^2\abs{y}^{-n-\sig-2}dy +
\lam\int_{\real^n}
(2-\sig)(\Gamma(x_0-y)-\Gamma(x_0))y_1^2\abs{y}^{-n-\sig-2}dy\\
&\geq \lam\int_{A^c_{x_0}}
(2-\sig)(\Gamma(x_0+y)-\Gamma(x_0))y_1^2\abs{y}^{-n-\sig-2}dy\\
&\geq \lam\sum_{k^*} \int_{A^c_{x_0}\intersect R_{k^*}}
(2-\sig)(\Gamma(x_0+y)-\Gamma(x_0))y_1^2\abs{y}^{-n-\sig-2}dy\\
&\geq \lam\sum_{k^*} \int_{A^c_{x_0}\intersect R_{k*}}
(2-\sig)f(x_0)\abs{y}^{\sig} y_1^2\abs{y}^{-n-\sig-2}dy.
\end{align*}

Since $\abs{A^c_{x_0}\intersect R_{k^*} }\geq |R_k|/2$, it can be checked easily that
\begin{equation*}
\int_{A^c_{x_0}\intersect R_{k^*}}y_1^2|y|^{-n-2}dy	\geq c_n\;\;\forall k^*
\end{equation*}
which together with the previous inequality means that
\begin{align*}
f(x_0)&\geq \lam f(x_0)\sum_{k^*}
(2-\sig) c_n\abs{A^c_{x_0}\intersect R_{k^*} }
r_{k^*+1}^{-n+\sig-\sig} \\
&\geq \lam f(x_0)\sum_{k^*+1}
(2-\sig)c_n\om(n)(1-2^{-n})\\
&\geq \lam f(x_0)(2-\sig)\frac{1}{2}\lam c_n\om_n(\#k^*)
\end{align*}

Therefore, we conclude
\begin{equation*}
(\# k^*)\leq
\frac{2}{(2-\sig)\lam}C_n
\end{equation*}

\end{proof}

\begin{rem}
It is worth comparing Lemma \ref{lem:GoodRing} to \cite[Lemma 8.1]{CaSi-09RegularityIntegroDiff}, where a decay rate of $\Gamma$ away from its supporting hyperplane of $r^2$ was used instead of the $r^\sig$ used here.  Using $r^\sig$ seems to give more precise information for which rings the decay estimate can fail.
\end{rem}

\begin{prop}\label{prop:PointToMeasureSIGMA}
Let $x_0$ be such that $\Gamma(x_0)=\inf_{\real^n}\{\Gamma\}$ and
$\sig\in(0,2)$.  Then
\begin{equation*}
-\inf_{\real^n} \{P\}=-\inf_{B_3}\{P\}\geq
C(n,\lambda)(-\Gamma(x_0))^{2/\sig}\left(\frac{1}{2f(x_0)}\right)^{(2-\sig)/\sig}.
\end{equation*}
\end{prop}

\begin{proof}[Proof of Proposition \ref{prop:PointToMeasureSIGMA}]
First we note the equality between the infimum of $P$ in all of $\real^n$ and in $B_3$ comes from Corollary \ref{cor:InfPB3}. 
Once we choose an appropriate $\rho_0$ from Lemma \ref{lem:GoodRing},
this proposition will follow directly from the definition of $P$.  We
would like to restrict the beginning radius, $\rho_0$, so that
whenever $y\in A_{x_0}\intersect R_{\rho_0}$, $y$ also satisfies
(recall $\Gamma$ is always negative)
\begin{equation*}
\Gamma(x_0+y)\leq \frac{1}{2}\Gamma(x_0).
\end{equation*}
This is guaranteed if we let 
\begin{equation*}
\rho_0
=\left(\frac{-\Gamma(x_0)}{2f(x_0)}\right)^{1/\sig},
\end{equation*}
which gives for
$\abs{y}\leq \rho_0$ and $y\in A_{x_0}$
\begin{equation*}
\Gamma(x_0+y)-\Gamma(x_0)\leq f(x_0)\abs{y}^{\sig} \leq \frac{-1}{2}\Gamma(x_0).
\end{equation*}
Also in the estimates to follow, it will be very useful to observe
that if $\{k_g\}$ are the collection of indices resulting in good rings
satisfying
\begin{equation*}
\abs{A_{x_0}\intersect R_k}\geq \frac{1}{2}\abs{R_k},
\end{equation*}
and $N(x_0)$ is the upper bound on the bad indices, $(\#k^*)$, where
the collection $\{k^*\}$ are the ones satisfying  (\ref{eq:BadRing}),
then
\begin{equation*}
\sum_{k_g} r_{k_g}^{2-\sig}\geq \sum_{k\geq N(x_0)} r_k^{2-\sig}.
\end{equation*}
Recall also that
\begin{equation*}
N(x_0)= \frac{2}{(2-\sig)\lam\omega_n}.
\end{equation*}

Now we have
\begin{align*}
-P(x_0) &= A(n,2-\sig)\int_{\real^n}-\Gamma(x_0+y)\abs{y}^{-n+2-\sig}dy\\
&\geq A(n,2-\sig)\sum_{k_g}\int_{A_{x_0}\intersect
R_{k_g}}-\frac{1}{2}\Gamma(x_0)\abs{y}^{-n+2-\sig}dy\\
&\geq A(n,2-\sig)(-\frac{1}{2}\Gamma(x_0))\sum_{k_g}\abs{A\intersect
R_{k_g}}r_{k_g+1}^{-n+2-\sig}\\
&\geq A(n,2-\sig)(-\Gamma(x_0))C(n)\sum_{k_g}r_{k_g}^{2-\sig}\\
&\geq A(n,2-\sig)(-\Gamma(x_0))C(n)\sum_{k\geq N(x_0)}r_{k}^{2-\sig}\\
&= A(n,2-\sig)(-\Gamma(x_0))C(n)\rho_0^{2-\sig}\sum_{k\geq
N(x_0)}2^{-k(2-\sig)}\\
&= A(n,2-\sig)(-\Gamma(x_0))C(n)\rho_0^{2-\sig}(2^{-(2-\sig)})^{N(x_0)}\sum_{k\geq0}
(2^{-(2-\sig)})^k\\
&\geq A(n,2-\sig)(-\Gamma(x_0))C(n)\rho_0^{2-\sig}(2^{-2/(\lam
\omega_n)})(1-2^{-(2-\sig)})^{-1}.
\end{align*}
In the calculations above, the ``constant'', $C(n)$, was used multiple
times as different numbers to absorb any extra constants which only
depended on the dimension.

At this point we can collect the various information about $\rho_0$, $N(x_0)$. Observe that
\begin{equation*}
\lim_{\sig\to2}A(n,2-\sig)(1-2^{(2-\sig)})^{-1}=C(n),
\end{equation*}
which follows from the fact that $A(n,2-\sig) \sim 2-\sig$ and $2^{2-\sig} \sim 1$ as $\sig \to 2$. We conclude
\begin{equation*}
-P(x_0)\geq C(n)2^{-2/(\lam \omega_n)}(-\Gamma(x_0))\rho_0^{(2-\sig)}.
\end{equation*}
Recall that $\rho_0= \left ( \frac{-\Gamma(x_0)}{2f(x_0)}\right
)^{1/\sig}$ and that trivially $-P(x_0)\leq -\inf P$
so that
\begin{equation*}
-\inf P \geq C(n)2^{-2/(\lam \omega_n)}(-\Gamma(x_0))\left ( \frac{-\Gamma(x_0)}{2f(x_0)}\right
)^{(2-\sig)/\sig},
\end{equation*}
which proves the proposition.
\end{proof}

\begin{rem}
There is a strange but straightforward phenomenon for $\Gamma$ resulting from the nonlocal nature of $M^-$.  As above, suppose that $\Gamma(x_0)=\inf\{\Gamma\}$; note $x_0\in B_1$ by the assumptions on $u$.  By the regularity of $\Gamma$, we can evaluate $M^-(\Gamma,x_0)$ (and also recall $\Gamma=0$ outside of $B_3$), and we see that for an appropriate $A$ with $\Tr(A)\geq \lam$
\begin{align*}
f(x_0)&\geq M^-(\Gamma.x_0)\geq (2-\sig)\int_{\real^n}\del\Gamma(x_0,y)y^TAy\abs{y}^{-n-\sig-2}dy\\
&> (2-\sig)\int_{\abs{y}>4}\del\Gamma(x_0,y)y^TAy\abs{y}^{-n-\sig-2}dy
\geq -\Gamma(x_0)C(n,\lam,\sig).
\end{align*}
This gives the very strange conclusion that 
\begin{equation*}
f(x_0)> -C(n,\lam,\sig)\Gamma(x_0),
\end{equation*}
so long as $\sig<2$, which is purely an artifact of the nonlocal nature of the operators (this information is not preserved as $\sig\to2$).  Furthermore, this gives good confirmation that the choice of $\rho_0$ in Proposition \ref{prop:PointToMeasureSIGMA} will never be too large.
\end{rem}

There is another way of stating Proposition \ref{prop:PointToMeasureSIGMA} which allows for a more intuitive interpretation as an interpolation result. 
\begin{cor}There exists $C(n,\lambda)>0$ such that if $\sigma,\Gamma$
and $P$ are all as above, then:
\begin{equation*}
\|\Gamma\|_\infty       \leq
C(n)\|P\|_\infty^{\sig/2}\;\|f\|_\infty^{(2-\sig)/2}
\end{equation*}
\end{cor}

\noindent
Written in this form, the nature of Proposition
\ref{prop:PointToMeasureSIGMA} as an interpolation estimate becomes
clear. To see why such an estimate is to be expected, we prove a
closely related result.

\begin{prop}[Interpolation; also as Lemma 4.1 of \cite{Silv-07}]\label{prop:StandardInterpolationEstimate} Suppose
$\sigma,\Gamma$ and $P$ are as before, then if $P$ is
$C^{1,1}(\real^n)$ we have
\begin{equation*}
\|\Gamma\|_\infty       \leq
C(n,\lambda)\|P\|_\infty^{\sig/2}\;[P]_{C^{1,1}}^{(2-\sig)/2}
\end{equation*}
\end{prop}

\begin{proof}[Proof of Proposition \ref{prop:StandardInterpolationEstimate}]
The proof is tailored after interpolation estimates for H\"older norms
(for similar, but not identical statements, the interested reader
should see \cite[Chapter 6, Appendix]{GiTr-98}).

Since $\Gamma =\Laplace{(2-\sigma)/2} P$ we always have (with
$A=\abs{A(n,-(2-\sigma))}$ )
\begin{equation*}
|\Gamma (x)|= A\left | \int_{\real^n}\frac{\delta
P(x,y)}{|y|^{n+2-\sigma}}dy \right |
\end{equation*}
\begin{equation}\label{eqn:InterpolationProp1}\leq A\left |
\int_{B_\rho}\frac{\delta P (x,y)}{|y|^{n+2-\sigma}}dy \right | +
A\left | \int_{B_\rho^c}\frac{\delta P(x,y)}{|y|^{n+2-\sigma}}dy
\right |\;\;\forall \rho>0.
\end{equation}

\noindent
The second integral can be controlled as follows
\begin{equation*}
\left | \int_{B_\rho^c}\frac{\delta P(x,y)}{|y|^{n+2-\sigma}}dy \right
| \leq  2\|P\|_{L^\infty(\real^n)}
\int_{B_\rho^c}\frac{1}{|y|^{n+2-\sigma}}dy=    2
\omega_n\|P\|_{L^\infty(\real^n)}\int_\rho^{+\infty} r^{-3+\sigma}dr
\end{equation*}
\begin{equation}\label{eqn:InterpolationProp2}
=\|P\|_{L^\infty(\real^n)}\frac{2\omega_n}{2-\sigma}\rho^{-(2-\sigma)}.
\end{equation}

\noindent
For the first integral we have
\begin{equation*}
\left | \int_{B_\rho}\frac{\delta P(x,y)}{|y|^{n+2-\sigma}}dy \right |
\leq  \left ( \sup \limits_{|y|\leq \rho} \frac{|\delta
P(x,y)|}{|y|^2}\right ) \int_{B_\rho}\frac{1}{|y|^{n-\sigma}}dy
\end{equation*}
\begin{equation}\label{eqn:InterpolationProp3}
=\left ( \sup \limits_{|y|\leq \rho} \frac{|\delta
P(x,y)|}{|y|^2}\right ) \frac{\omega_n}{\sigma}\rho^{\sigma}\leq
[P]_{C^{1,1}(\real^n)}\frac{\omega_n}{\sigma}\rho^{\sigma}.
\end{equation}

\noindent
We may now plug inequalities (\ref{eqn:InterpolationProp2}) and
(\ref{eqn:InterpolationProp3})
in (\ref{eqn:InterpolationProp1}) to obtain:
\begin{equation}\label{eqn:InterpolationProp4}
|\Gamma(x)|\leq A\omega_n\left (
\frac{2}{2-\sigma}\|P\|_{L^\infty(\real^n)}\rho^{-(2-\sigma)}+
\frac{C_n[P]_{C^{1,1}}(\real^n)}{\sigma} \rho^\sigma\right )
\;\;\;\;\forall \rho>0.
\end{equation}

\noindent
Then we may pick the $\rho>0$ that minimizes the  right hand side  of
(\ref{eqn:InterpolationProp4}). Since we have a convex function of the
parameter $\rho$ we only need to get its critical point:
\begin{equation*}
\frac{d}{d\rho} \left  (
\frac{a}{2-\sigma}\rho^{-(2-\sigma)}+\frac{b}{\sigma}\rho^\sigma        \right
)=0 \Rightarrow a\rho^{-(3-\sigma)}=b \rho^{\sigma-1}
\Rightarrow \rho=\sqrt{a/b}.
\end{equation*}
Finally, putting $\rho=\sqrt{a/b}$ gives us the minimum value (basically we used Young's inequality)
\begin{equation*}
\frac{a}{2-\sigma}a^{-(2-\sigma)/2}b^{(2-\sigma)/2}+\frac{b}{\sigma}a^{\sigma/2}b^{-\sigma/2}=\frac{2}{(2-\sigma)\sigma}a^{\sigma/2}b^{(2-\sigma)/2}.
\end{equation*}

\noindent
Where $a=2\|P\|_{L^\infty(\real^n)}$ and
$b=C_n[P]_{C^{1,1}(\real^n)}$, going back to (\ref{eqn:InterpolationProp4})
we obtain
\begin{equation*}
|\Gamma(x)| \leq \frac{2^{1+\sigma/2}\omega_n C_n^{(2-\sigma)/2}
A(n,-(2-\sigma))}{(2-\sigma)\sigma
}\|P\|_{L^\infty(\real^n)}^{\sigma/2}[P]_{C^{1,1}(\real^n)}^{(2-\sigma)/2}.
\end{equation*}

The proof is now done, as the term on the right hand side stays
uniformly bounded for $\sigma \in (0,2)$, which is thanks to the
behavior of $A(n,-(2-\sigma))$ as $\sigma\to 2$.
\end{proof}

Before we conclude this section, we must prove one last straightforward feature of the potential, $P$.  We must be able to compare the values of $P$ on the boundary of some ball, $B_R$, and make sure that the difference between the values on the boundary and the infimum of $P$ are still comparable to the infimum itself. 

\begin{lem}[Decay  of The Potential]\label{lem:PotenialOscillation}
There exists a radius, $R_\sig$, depending on $\sig$ and $n$ such that 
\begin{equation*}
-\inf_{B_{R_\sig}}\{P\} + \inf_{\partial B_{R_\sig}}\{P\} \geq \frac{1}{2} \left(-\inf_{B_{R_\sig}}\{P\}\right) .
\end{equation*}
\end{lem}

\begin{proof}[Proof of Lemma \ref{lem:PotenialOscillation}]
Let $x\in\partial B_R$ and $x_0$ be such that $-P(x_0)=-\inf(P)$.  We will first make a computation with a generic $R$ and then choose the value for $R_\sig$ at the end.  Lemma \ref{lem:GammaCompactSupport} tells us that $\Gamma=0$ in $\real^n\setminus B_3$.  Therefore we can estimate the integral
\begin{align*}
-P(x)&=A(n,2-\sig)\int_{\real^n}-\Gamma(x+y)\abs{y}^{-n+2-\sig}dy\\
&= A(n,2-\sig)\int_{x+y\in B_3}-\Gamma(x+y)\abs{y}^{-n+2-\sig}dy\\
&= A(n,2-\sig)\int_{x_0+w\in B_3}-\Gamma(x_0+w)\abs{x_0-x+w}^{-n+2-\sig}dw\\
&
=A(n,2-\sig)\int_{x_0+w\in B_3}-\Gamma(x_0+w)\abs{w}^{-n+2-\sig}\left(\frac{\abs{x_0-x+w}}{\abs{w}}\right)^{-n+2-\sig}dw
\\
&\leq A(n,2-\sig)\int_{x_0+w\in B_3}-\Gamma(x_0+w)\abs{w}^{-n+2-\sig}\left(\frac{R-3}{6}\right)^{-n+2-\sig}dw\\
&= P(x_0)\left(\frac{R-3}{6}\right)^{-n+2-\sig}.
\end{align*}
Now, choosing $R_\sig$ so that
\begin{equation*}
\left(\frac{R_\sig-3}{6}\right)^{-n+2-\sig}\leq \frac{1}{2}
\end{equation*}
gives the result.
\end{proof}

\section{$\det(D^2P)$ as an Integro-Differential Operator on $\Gamma$}\label{sec:DetD2P}
\setcounter{equation}{0}
The point of using the potential, $P$, is that we can convert our fractional order equation into a 2nd order equation and appeal to known results.  In this way we would like also to recognize the operation $\det(D^2P)$ as a $\sig$-order operator on $\Gamma$, which one can hope will be comparable to $M^-$ by below (proved in Lemma \ref{lem:NonlocalDetMminusOrder}).  This comparison would then allow us to bring together the geometric result involving $\inf(P)$ and the envelope property that $M^-(\Gamma)\leq Cf$ on the set $K_u$.  Fundamental to this pursuit is the formula for the determinant: 
\begin{equation*}
\left(\det(B)\right)^{1/n} = \frac{1}{n}\inf\{\Tr(AB) : A\geq0\ \text{and}\ \det(A)=1\}\ \  \mbox{whenever }B \geq 0.
\end{equation*}
The interested reader should consult \cite{CaNiSp-85} and \cite{Gard-59} for further discussion.  The integro-differential nature of $\det(D^2P)$ acting on $\Gamma$ is developed by making rigorous the formal computation of convolving $\Gamma$ with the derivatives of the Riesz Kernel.

\begin{lem}[Derivatives of The Potential]\label{lem:SingularIntegralPij}
Assume that for a.e. $x\in\Om\subset\subset B_3$, 
$$0\leq h_\sig(\Gamma,x)\leq C_1\ \ \text{and}\ \ \norm{\Gamma}_{H^\sig(\Om)}\leq C_1,$$
where $C_1$ is allowed to depend on $\Om$.  
Then $P \in C^{1,1}(\Om)$ and  the second derivative of $P$ at a.e. $x \in \Om$ in terms of $\Gamma$ is
\begin{align}
&D^2P(x) = \nonumber \\ 
&\frac{(n+\sig-2)(n+\sig)}{2}\int_{\real^n} A(n,2-\sig) \del \Gamma(x,y)  \left[\frac{y\otimes y}{\abs{y}^{n+\sig+2}}-
                                                                          \frac{\Id}{(n+\sig)\abs{y}^{n+\sig}}\right]dy.
\label{eq:Pxixj}
\end{align}
\end{lem}

\begin{rem}
It is worth pointing out explicitly that Lemma \ref{lem:SingularIntegralPij} gives enough regularity for our purposes later to work with $P$ directly in any estimates involving $D^2P$.
\end{rem}

\begin{rem}\label{rem:DerivativesCompactFormula} In particular, our definition of $h_\sig(v,x)$ gives
\begin{equation*}
D^2 ( v * K_{2-\sig}) = h_\sig(v,x) +\tfrac{(-\Delta)^{\sig/2}v(x)}{n+\sig}\;\Id.
\end{equation*}
\end{rem}

\begin{proof}[Proof of Lemma \ref{lem:SingularIntegralPij}]
We consider as a regularization of $P$, the convolution of $\Gamma$ with a modified Riesz Kernel $K^\al_{2-\sigma}$, where for $\alpha>0$ we define
\begin{equation*}
K^\alpha_{2-\sigma}(y) := \left \{ \begin{array}{rl}
K_{2-\sigma}(y) & \mbox{ if } |y|\geq \alpha\\
Q_\alpha(y) & \mbox{ if } |y| \leq \alpha
\end{array}\right.	
\end{equation*}
Here $Q_\alpha(y)$ is the only quadratic polynomial that makes $K^\alpha_{2-\sigma}$ as defined above a $C^{1,1}(\real^n)$ function across $\partial B_\alpha$, 
\begin{equation*}
Q_\alpha(y) = \frac{A(n,2-\sigma)}{2\alpha^{n+\sigma}} \left ( -(n+\sigma-2)|y|^2+(n+\sigma)\alpha^2\right ).	
\end{equation*}
It will be handy to have the formula for the second derivatives of $K^\al_{2-\sigma}$ away from $\partial B_\alpha$
\begin{equation}\label{eq:KAlDerivativesFormula}
D^2K^\alpha_{2-\sigma}(y)= \left \{ \begin{array}{rl}
A(n,2-\sig)|y|^{-n-\sig}(n+\sig-2)\left((n+\sig)\frac{y\; \otimes \;y}{\abs{y}^2}-\Id\right) & \mbox{ if }  |y|>\alpha\\
\\
-\frac{A(n,2-\sig)(n+\sig-2)}{\al^{n+\sig}} \Id & \mbox{ if } |y|<\alpha
\end{array}\right.
\end{equation}
We also use the elementary identities (which follow by reflection and rotation symmetry):
\begin{equation}\label{eq:PDerivativesZeroCrossTermsIntegral}
\int_{\partial B_r}y_iy_jdS(y)=0 \;\; \mbox{ for } i\neq j,\;\;\int_{\partial B_r}y_i^2dS(y)=\frac{r^2}{n}\int_{\partial B_r}dS(y)\;\;\forall r>0
\end{equation}

We denote the regularization of $P$ as $P_\al:=\Gamma*K^\al_{2-\sig}$, note that $P_\al\in C^{1,1}(\real^n)$ for each $\al>0$. To obtain (\ref{eq:Pxixj}), we will compute the derivatives using (\ref{eq:KAlDerivativesFormula}) inside of the convolution.  Then, imitating a classical analytic continuation argument \cite[p.45-47]{Land-72} we will pass the formula for $D^2P_\alpha$ to the limit $\al \to 0^+$. First, by the reflection symmetry of $K^\al_{2-\sig}$ about $0$ we may write
\begin{equation}\label{eqn: D2Pstepone}
D^2P_\al(x) = 	\frac{1}{2}\int_{\real^n} \left (\Gamma(x+y)+\Gamma(x-y) \right ) D^2K^\al_{2-\sig}(y)dy
\end{equation}

The classical argument relies on computing the average $D^2K^\al_{2-\sig}$ in and out of $B_\al$, we begin by noting that according to (\ref{eq:KAlDerivativesFormula})
\begin{equation*}
\int_{B_\al}D^2K^\al_{2-\sig}(y)dy = -A(n,2-\sig)(n+\sig-2)\frac{|B_1|}{\al^\sig}\Id.
\end{equation*}

\noindent
Next, using integration in spherical coordinates together with (\ref{eq:PDerivativesZeroCrossTermsIntegral}) it is easy to check that
\begin{align*}
\int_{B_\al^c}D^2K^\al_{2-\sig}(y)dy &= A(n,2-\sig)(n+\sig-2)\int_{B_\al^c}|y|^{-n-\sig}\left ((n+\sig)\frac{y\otimes y}{\abs{y}^2}-\Id \right ) \; dy\\
&=A(n,2-\sig)(n+\sig-2)\int_{B_\al^c}|y|^{-n-\sig}\left (\frac{(n+\sig)}{n}\Id-\Id  \right ) \; dy\\
&=A(n,2-\sig)(n+\sig-2)\frac{|B_1|}{\al^\sig}\Id. 
\end{align*}

\noindent
Thus the integral of $D^2K^\al_{2-\sig}$ in and out of $B_\al$ is the same except for a sign, in other words
\begin{equation*}
\int_{\real^n} D^2K^\al_{2-\sig}(y)dy=0\;\;\forall \alpha>0.
\end{equation*}

This means that the identity (\ref{eqn: D2Pstepone}) still holds if we add zero to the right hand side as follows
\begin{align*}
D^2P_\al(x) &= \frac{1}{2}\int_{\real^n} \left (\Gamma(x+y)+\Gamma(x-y) \right ) D^2K^\al_{2-\sig}(y)dy-\Gamma(x)\int_{\real^n} D^2K^\al_{2-\sig}(y)dy,
\end{align*}
in other words, 
\begin{equation*}
D^2P_\al(x) = \frac{1}{2}\int_{\real^n} \delta \Gamma(x,y) D^2K^\al_{2-\sig}(y)dy, \;\;\forall \;\al>0.
\end{equation*}

\noindent
This, together with the pointwise bound
\begin{equation*}
|D^2K^\al_{2-\sig}(y)|\leq A(n,2-\sig)C(n)|y|^{-n-\sig},\;\;\;\forall y\; \in \real^n,\;\alpha>0,
\end{equation*} 
implies for all $x\in B_3$ that
\begin{equation}\label{eq:D2PAlphaBound}
|D^2P_\al(x)|\leq A(n,2-\sig)C(n)\int_{\real^n} \frac{|\delta \Gamma(x,y)|}{|y|^{n+\sig}}dy.	
\end{equation}

We now make a few remarks to conclude the lemma.  First, by construction of $K^\al_{2-\sig}$, it follows that $P^\al\to P$ uniformly.  Second, (\ref{eq:D2PAlphaBound}) combined with the assumption that $\Gamma\in H^\sig_{loc}$ imply that all of the hessians, $D^2P_\al$, are uniformly bounded in $L^2(B_R)$ for any $B_R\subset\subset B_3$; furthermore by dominated convergence, $D^2P_\al(x)\to D^2P(x)$ pointwise for all $x$ such that 
\[
\int_{\real^n} \frac{|\delta \Gamma(x,y)|}{|y|^{n+\sig}}dy<\infty,
\]
which is a.e. $x$ in $B_3$.  Third, after extracting a $H^{2}_{loc}$ weakly convergent subsequence from $P_\al$ and using uniqueness of limits, we can conclude that $P$ indeed is $H^2_{loc}(B_3)$ and (\ref{eq:Pxixj}) for $D^2P$ indeed holds.

Finally, we note that once $D^2P$ can be computed a.e. $x$ by (\ref{eq:Pxixj}), the bound on $h_\sig(\Gamma,x)$ immediately translates to a $C^{1,1}(\Om)$ bound for $D^2P$.

\end{proof}

Now that we have a formula for $D^2P$,we can compute the rest of the operators needed for $\det(D^2P)$.

\begin{lem}[Integro-Differential Form of $\Tr(AD^2P)$]\label{lem:SingularIntegralTraceAHessionP}
For any positive matrix $A$ and $\sigma \in (0,2)$ we define 

\begin{equation}\label{eq:ASigma}
A_\sigma:= A+\frac{\Tr(A)}{\sigma}I. 
\end{equation}
Then $\Tr(A_\sigma D^2P)$ defines an elliptic integro-differential operator acting on $\Gamma$. Moreover, it is given by the formula
\begin{align}
&\Tr(A_\sigma D^2P)(x) = \nonumber \\ 
&\frac{(n+\sig-2)(n+\sig)}{2}A(n,2-\sig)\int_{\real^n} \del \Gamma(x,y) \frac{y^TAy}{\abs{y}^{n+\sig+2}}dy.
\label{eq:TraceD2POnGamma}
\end{align}
\end{lem}

\begin{proof}[Proof of Lemma \ref{lem:SingularIntegralTraceAHessionP}]
This is a direct consequence of Lemma \ref{lem:SingularIntegralPij} which says
\begin{align}
&\Tr(A_\sigma D^2P)(x) = \nonumber \\ 
&\frac{(n+\sig-2)(n+\sig)}{2}A(n,2-\sig)\int_{\real^n} \del \Gamma(x,y) \left[\frac{y^TA_\sigma y}{\abs{y}^{n+\sig+2}}-                                                                          \frac{\Tr(A_\sigma)}{(n+\sig)\abs{y}^{n+\sig}}\right]dy. \nonumber
\end{align}
Since $y^T A_\sigma y = y^T A y +\frac{\Tr(A)}{\sigma}|y|^2$ and $\frac{1}{n+\sigma}\Tr(A_\sigma)=\frac{\Tr(A)}{\sigma}$ this integral equals (\ref{eq:TraceD2POnGamma}).

\end{proof}

\begin{rem}\label{rem:ASigmaElliptic}
Observe that the matrix $A_\sigma$ becomes \emph{strictly elliptic}, that is, $A_\sigma \geq \lambda \Id$ as soon as $\Tr(A)\geq \sigma \lambda$.  It is thanks to this fact that our estimates apply to equations with kernels that are not strictly positive. The authors would like to acknowledge the anonymous referee for this important observation as well as their suggestion of introducing the change of variables $A \to A_\sigma$ which significantly helps with the presentation.
\end{rem}

\begin{DEF}[Integro-Differential Form of $\det(D^2P)$]\label{def:DetSingularIntegral}
The operator $\Gamma \to \det(D^2P)(x)$, defines a nonlinear integro-differential operator acting on $\Gamma$.  At the points, $x$, where
$$\int_{\real^n}\frac{\abs{\del\Gamma(x,y)}}{\abs{y}^{n+\sig}}dy<\infty\ \ \text{and}\ \ D^2P(x) \geq 0,$$ 
this definition agrees with 
\begin{equation*}
\D_\sig(\Gamma,x)=\left(\frac{1}{n}\inf\left\{\T_A(\Gamma,x):\ A\geq0\ \text{and}\ \det(A)=1\right\}\right)^n=\det(D^2P),
\end{equation*}
where we use the notation $\T_A(\Gamma,x):= \Tr(A D^2P(x))$.
\end{DEF}

\begin{rem}\label{rem:DsigmaPeculiar}
It is worthwhile to remark on the peculiar form of formulas (\ref{eq:Pxixj}) and (\ref{eq:TraceD2POnGamma}). Note that the
kernel 
\begin{equation*}\frac{y^TBy}{\abs{y}^{n+\sig+2}}-\frac{\Tr(B)}{(n+\sig)\abs{y}^{n+\sig}}
\end{equation*} might take negative values in some directions even for matrices such that $B\geq 0$, so the kernels appearing in the
formulas are not necessarily positive. The kernels will be positive for matrices of the form $A_\sigma:= A+\frac{\Tr(A)}{\sigma}I,\;\; A\geq 0$ which is why they were introduced in Lemma \ref{lem:SingularIntegralTraceAHessionP}.  However, matrices which are not of this form appear in the formula for the operator $\mathcal{D_\sigma}$, which means that  $\D_\sigma(\Gamma,x)=\det (D^2P(x))$ is not an elliptic operator for a general $\Gamma$ (not even if $E_\sigma(\Gamma)\geq 0$). Despite these issues, $\D_\sig$ serves a crucial purpose in the proof of Theorem \ref{thm:Main} as suggested by the the next two Lemmas. 
\end{rem}

\begin{lem}[Nonlocal Determinant - Minimal Operator Ordering]\label{lem:NonlocalDetMminusOrder} 
If
$\ds\int_{\real^n}\frac{\abs{\del\Gamma(x,y)}}{\abs{y}^{n+\sig}}dy<\infty$ 
and $E_\sig(\Gamma,x)\geq0$, then whenever $D^2P(x)\geq0$, it also holds that 
\begin{equation*}
\D_\sig(\Gamma,x)\leq \frac{C(n)}{\lam^n} (M^-(\Gamma,x))^n.
\end{equation*}
\end{lem}

\begin{proof}[Proof of Lemma \ref{lem:NonlocalDetMminusOrder}]

Thanks to the assumption of $D^2P(x) \geq 0$,\  we have 
\begin{align*}
\left(\D_\sig(\Gamma,x)\right)^{1/n}&\leq\inf_{\det(A)=1}\left\{\frac{1}{n}\T_A(\Gamma,x)\right\}\\
&\leq \inf_{\det(A_\sigma)=1,A\geq 0}\left\{\frac{1}{n}\T_{A_\sigma}(\Gamma,x)\right\} \\
&= \inf_{\lambda \Id \leq A_\sigma,A\geq 0}\left\{\frac{1}{n \;\det(A_\sigma)^{1/n}}\T_{A_\sigma}(\Gamma,x)\right\} \\
\end{align*}
Now, by definition (see (\ref{eq:ASigma})) of $A_\sigma$ in we have the set inclusion (see also Remark \ref{rem:ASigmaElliptic})
\begin{equation*}
\{A_\sigma :\;\; A\geq0 ,\;\; A_\sigma \geq \lambda\Id,A \}	 \supset \{ A_\sigma: \;\; A\geq 0,\;\;\Tr(A)\geq \sigma\lambda\}
\end{equation*}
Which together with the previous inequalities shows that
\begin{equation*}
\left(\D_\sig(\Gamma,x)\right)^{1/n}\leq \inf_{A\geq 0,\;\Tr(A)\geq \sigma\lambda}\left\{\frac{C(n)}{\lambda}A(n,2-\sigma)\int_{\real^n}\delta\Gamma(x,y)\frac{y^TAy}{\abs{y}^{n+\sigma+2}}dy\right\}
\end{equation*}
and the Lemma is proved.
\end{proof}

\begin{rem}
It is worth remarking that Lemma \ref{lem:NonlocalDetMminusOrder} is one of the three key reasons for a good definition of the envelope, $\Gam$.  In fact, this ordering between $\D_\sig$ and $M^-$ is not to be expected for a general function, $v$, which does not necessarily satisfy $E_\sig(v)\geq0$-- just as in the second order setting between the Monge-Amp\`ere and Minimal Pucci operators within the class of convex functions.
\end{rem}

The restriction to points $x$ where $D^2P(x) \geq 0$ is not too severe, as later on we will reduce to those $x$ for which $P$ coincides with its convex envelope, and there we shall always have $D^2P \geq 0$. Very much related to this fact is the following crucial Lemma which will also guarantee that we will only need to test the equation at points where $u=\Gamma$.

\begin{lem}[$\{\det(D^2P) > 0\}$ is contained in the contact set]\label{lem:NonlocalDetContactSet}
If for a.e. $x$, \[\ds\int_{\real^n}\frac{\abs{\del\Gamma(x,y)}}{\abs{y}^{n+\sig}}dy<\infty\] and $\Gamma \neq 0$, then for a.e. $x_0$ such that $\Gamma(x_0)\not=u(x_0)$, we can find some direction $\tau$ such that $P_{\tau\tau}(x_0) \leq 0$.  In particular,
\[
\left(\{\det(D^2P) > 0\}\intersect \{D^2P\geq0\}\right)\subset \{u=\Gamma\}.
\]
\end{lem}

\begin{proof}[Proof of Lemma \ref{lem:NonlocalDetContactSet}]
Lemma \ref{lem:GammaRegularityByAbove} plus standard techniques in viscosity solutions tell us that at almost every such $x_0$, $E_\sig(\Gamma,x_0)=0$.  Hence for each of these $x_0$ there exists a direction, \BlueComment{$\tau$},  such that
\begin{equation*}
(h_\sig)_{ij}(\Gamma,x_0)\tau_i\tau_j = 0 
\end{equation*}
Hence we can plug $\tau$ into the formula (\ref{eq:Pxixj})  for $P_{\tau\tau}$ and get (recall $\tau_i\tau_i = |\tau|^2=1$)

\begin{equation*}
P_{\tau\tau} =(h_\sig(\Gamma,x_0))_{ij}\tau_i\tau_j- C_{n,\sig}\int_{\real^n} \frac{\del \Gamma(x_0,y) }{\abs{y}^{n+\sig}}dy=- C_{n,\sig}\int_{\real^n} \frac{\del \Gamma(x_0,y) }{\abs{y}^{n+\sig}}dy\leq 0.
\end{equation*}
In this last inequality we have used the fact that $\Gamma$ solves $E_\sig(\Gamma)\geq0$ everywhere, and so because $-\Laplace{\sig/2}\Gamma = \Tr(h_\sig(\Gamma))$, we have the correct sign for the leftover term.
This concludes the first assertion of the lemma.

To finish with the set inclusion, we note from the Corollary \ref{cor:D2PNonNegSet}, that $\{D^2P\geq0\}\subset B_3$, and the fact that in $B_3\intersect \{u\not=\Gamma\}$, $\Gamma$ solves $E_\sig(\Gamma,x)=0$ (see (\ref{eq:AppendixObstacleViscosityEq})).  The set inclusion then follows by the first assertion of the lemma.
\end{proof}
\section{Proof of The Main Theorem}\label{sec:Proof}
\setcounter{equation}{0}

We are now in a position to prove the main theorem, Theorem \ref{thm:Main}.  The main step left to prove is an estimate on $\inf(P)$ which involves a measure theoretic norm of $f$.  This will come from the geometric set-up of the second order ABP estimate for $P$, thanks to the various results of section \ref{sec:DetD2P}.  

\begin{proof}[Proof of Theorem \ref{thm:Main}]
The theorem will first be proved under the assumption that $u$ is $C^{1,1}$ from above, and then we will remove this restriction at the end of the proof using the standard method of inf-convolution of $u$.  Therefore, assume $u$ is $C^{1,1}$ from above.

As mentioned in the lead-up to this proof, we will be applying the Aleksandrov Estimate for convex functions to the convex envelope of $P$, but a crucial factor to be determined is the domain in which the argument will be applied.  We must work with an appropriately large radius according to Lemma \ref{lem:PotenialOscillation} in order to make sure that the difference of the values of $P$ on $\partial B_{R_\sig}$ and $\inf (P)$ are still comparable to $\inf_{\real^n} (P)$.  Therefore we are working in $B_{R_\sig}$ from Lemma \ref{lem:PotenialOscillation}.

Let $P^{CE}$ be the convex envelope of $P$ in $B_{R_\sig}$.  It is essential to note a very important consequence of Lemma \ref{lem:PCantTouchPlane} is that the contact set, $\{x:P(x)=P^{CE}(x)\}$, is contained within $B_3$.  Furthermore, since $P\in C^{1,1}_{loc}(B_3)$, we know that $P^{CE}\in C^{1,1}_{loc}(B_{R_\sig})$ (see \cite[Chapter 3]{CaCa-95}).  The Aleksandrov Estimate for convex functions (see \cite[Chapter 3]{CaCa-95}) then implies
\begin{align*}
-\inf_{B_{R_\sig}}\{P\} &\leq \left(-\inf_{\partial B_{R_\sig}}\{P\}\right) + C(n)\left(\int_{\{P=P^{CE}\}} \det(D^2P^{CE}(x))dx\right)^{1/n}\\
&\leq \left(-\inf_{\partial B_{R_\sig}}\{P\}\right) + C(n)\left(\int_{\{P=P^{CE}\}} \det(D^2P(x))dx\right)^{1/n}\\
&= \left(-\inf_{\partial B_{R_\sig}}\{P\}\right) + C(n)\left(\int_{\{P=P^{CE}\}} \D_\sig(\Gamma,x)dx\right)^{1/n}.
\end{align*}
\noindent
We note that the constant $C(n)$ in fact depends upon $R_\sig$, but it is uniformly bounded from below and above for our $R_\sig$, and so we simply denote it as a dimensional constant.

Thanks to the comparison between $\D_\sig$ and $M^-$, Lemmas \ref{lem:NonlocalDetMminusOrder} and \ref{lem:NonlocalDetContactSet}, and the comparison principle between $\Gamma$ and $u$ for $M^-$ on $K_u$ all combine to give
\begin{align}
-\inf_{B_{R_\sig}}\{P\} &\leq \left(-\inf_{\partial B_{R_\sig}}\{P\}\right) + C(n)\left(\int_{K_u} \D_\sig(\Gamma,x)dx\right)^{1/n}\nonumber \\
&\leq\left(-\inf_{\partial B_{R_\sig}}\{P\}\right) + \frac{C(n)}{\lam}\left(\int_{K_u}\left(M^-(\Gamma,x)\right)^{n} dx\right)^{1/n}\nonumber\\
&\leq\left(-\inf_{\partial B_{R_\sig}}\{P\}\right) + \frac{C(n)}{\lam}\left(\int_{K_u}\left(M^-(u,x)\right)^n dx\right)^{1/n}\nonumber\\
&\leq\left(-\inf_{\partial B_{R_\sig}}\{P\}\right) + \frac{C(n)}{\lam}\left(\int_{K_u}(f(x))^n dx\right)^{1/n}.\label{eq:ProofMainPestimate}
\end{align}

\noindent
We note that we have also used the $C^{1,1}$ from above nature of $u$ to evaluate (\ref{eq:PIDEmain}) classically.

In order to conclude the theorem in the case that $u$ is $C^{1,1}$ from above we must now recall how $\inf(P)$ relates to $\inf(\Gamma)$.  We can now combine the results of Proposition \ref{prop:PointToMeasureSIGMA} and Lemma \ref{lem:PotenialOscillation} with (\ref{eq:ProofMainPestimate}) to obtain (with obvious abuse of the dimensional constant, $C(n)$):
\begin{align*}
C(n,\lambda)(-\Gamma(x_0))^{2/\sig}\left(\frac{1}{2f(x_0)}\right)^{(2-\sig)/\sig}\leq \frac{C(n)}{\lam}\left(\int_{K_u}f^ndx\right)^{1/n},
\end{align*}
which after rearranging gives 
\begin{equation*}
-\inf\{\Gamma\}\leq \frac{C(n)}{\lam}(\norm{f}_{L^\infty(K_u)})^{(2-\sig)/2}(\norm{f}_{L^n(K_u)})^{\sig/2}.
\end{equation*}

Now to finish the theorem for a generic $u$, we work with the inf-convolution of $u$,
\begin{equation*}
u_\ep(x)=\inf_{\abs{x-y}\leq\sqrt{\ep}\norm{u}_\infty}\{u(y)+\frac{1}{\ep}\abs{x-y}^2\}
\end{equation*}
\noindent
(see \cite[Appendix A]{CaCrKoSw-96}, \cite[Appendix]{CrIsLi-92}, \cite[Equations (14), (15)]{JeLiSo-88UniquenessSecondOrder}, or \cite{LaLi-1986RegularizationHilbertSpace} for definitions and \cite[Appendix A]{CaCrKoSw-96} for the analogous argument regarding the ABP in the second order setting)
The main properties we will use are that $u_\ep$ is $C^{1,1}$ from above in $B_3$, that $u_\ep$ is \emph{increasing to $u$}, and that $u_\ep$ solves the equation (\ref{eq:PIDEmain}) with $f$ replaced by $f^\ep$ (see \cite[Lemma A.3]{CaCrKoSw-96} and \cite[Proposition 5.5]{CaSi-09RegularityIntegroDiff}) given as
\begin{equation*}
f^\ep(x)=\sup_{\abs{x-y}\leq\sqrt{\ep}\norm{u}_\infty}\{f(y)\}.
\end{equation*}
For $\Gamma_\ep$, $P_\ep$, and $K_{u_\ep}$ denoting all the corresponding operations using $u_\ep$ in place of $u$, (\ref{eq:ProofMainPestimate}) becomes
\begin{equation*}
-\inf_{\real^n}\{P_\ep\}\leq \frac{C(n)}{\lam}\left(\int_{K_{u_\ep}}(f^\ep)^ndx\right)^{1/n}.
\end{equation*}
So long as $\limsup_\ep f^\ep(x)\leq f(x)$, which is given by continuity of $f$ in this case,  we can conclude the estimate by letting $\ep\to0$ and using Lemma \ref{lem:ContactSetsLimsup}. 
\end{proof}

\begin{rem}
For future reference, we would like to collect a very important fact which is a cornerstone of the proof of Theorem \ref{thm:Main}.  It is the relationship between $\inf(\Gamma)$ and the integral of $\D_\sig(\Gamma)$,
\begin{equation}\label{eq:GammaDsigmaBound}
-\inf_{\real^n}(\Gamma)\leq \frac{C(n)}{\lam} \left(\int_{K_u}\D_\sig(\Gamma,x)dx\right)^{1/n},
\end{equation}
which holds whenever \[\ds\int_{\real^n}\frac{\abs{\del\Gamma(x,y)}}{\abs{y}^{n+\sig}}dy<\infty\]
for a.e. $x$.
\end{rem}

\section{Fractional Monge-Amp\`ere Type Operators}\label{sec:MongeAmpere}
\setcounter{equation}{0}

The Monge-Amp\`ere type operator $\D_\sig(\Gamma,x)$ is crucial to our
proof of Theorem \ref{thm:Main}, as it allows us to borrow the
divergence structure of the standard Monge-Amp\`ere which is essential
for the classical ABP Theorem. One may take the point of view that
$T_u(x)=\nabla P_u$ is a nonlocal gradient map given by $u$, and that
$\D_\sigma$ is just the Jacobian of the map $T_u$. Then, part of the
effort in our proof (Section \ref{sec:Pressure}) has been relating the
infimum of $\Gamma$ to the size of the image of the map $T_\Gamma(x)$.

We insist in referring to $\D_\sigma$ as a ``Monge-Amp\`ere type
operator" and not just the ``Monge-Amp\`ere operator", as it is not
clear that $\D_\sigma$ would be a definite, canonical analogue of this
operator for integro-differential equations. Note also that as $P_u$
is not necessarily convex the map $\nabla P_u$ might not be monotone,
and in fact the operator $\D_\sigma(u)$ might fail to be elliptic even
on the class of functions satisfying $E_\sigma\geq 0$ (for the
potential $P$ might not be convex). Another, perhaps more natural
candidate for an ``integro-differential Monge-Amp\`ere'' is given by
the operator
\begin{equation}\label{eqn:AlternateMongeAmpere}
\D^*_\sigma(u,x):= \left ( \inf \limits_{A\geq 0,\\ \det
A=1}\int_{\real^n}\frac{C(n)(2-\sigma)}{|A^{-1}y|^{n+\sigma}}\delta
u(x,y)dy \right )^{n_\sigma}
\end{equation}

Where $n_\sigma$ may be chosen to be $n_\sig\equiv n$ or at least such that
$n_\sigma \to n$ as $\sigma \to 2$. Unlike $\D_\sigma$, this second
operator is affine invariant and it can be checked easily that it is
well defined and (degenerate) elliptic in the class of functions which
are subsolutions of $E_\sigma = 0$. The main disadvantage of $\D^*_\sig$ with respect
to $\D_\sigma$ is that it is hard to relate directly the size
$\D^*_\sigma \Gamma $ to the infimum of $\Gamma$.  This relationship could be quantified if
$\D^*_\sigma$ had a connection to some sort of gradient map, which is (informally speaking) 
how the divergence structure of Monge-Amp\`ere contributes to the classical ABP 
(recall $\D_\sigma =\det (DT_u)$). 

The operator $\D^*_\sigma$ as defined in
(\ref{eqn:AlternateMongeAmpere}) would be a natural operator to
consider in the class of equations given by
\begin{equation}\label{eqn:AlternateLinear}
L_A(u,x)=\int_{\real^n}\frac{C(n)(2-\sigma)}{\det(A)|A^{-1}y|^{n+\sigma}}\delta
u(x,y)dy\;\;\;\text{ where } A^t=A,\;\lambda \text{I}\leq A^2 \leq
\Lambda \text{I}
\end{equation}

A posteriori, one can see the reasons for the proof of Theorem
\ref{thm:Main} being restricted to a much smaller class of equations
than those appearing in \cite{CaSi-09RegularityIntegroDiff}. There must
be a balance to ensure each one of Lemmas \ref{lem:NonlocalDetMminusOrder},
\ref{lem:NonlocalDetContactSet}, and
(\ref{eq:GammaDsigmaBound}) hold.  In principle, one could
create a new notion of determinant altogether, such that those three
meta-lemmas would still be valid for much richer families of kernels. The difficulty comes from the fact that
considering a larger family of kernels makes the operator $M^-$ more
extremal, and thus the question of finding a ``geometric'' operator
that is comparable $M^-$ becomes much harder to tackle.  In
conclusion, all these issues underline the lack of geometric equations
for integro-differential operators.
\section{Important Limits As $\sig\to2$}\label{sec:ImportantLimitSigmaTo2}
\setcounter{equation}{0}

In the series of works, \cite{CaSi-09EvansKrylov}, \cite{CaSi-09RegularityByApproximation}, and \cite{CaSi-09RegularityIntegroDiff}, all of the results were obtained in a fashion in which they are preserved as $\sig\to2$ and recover the corresponding results already proved for second order equations.  This has led to a unified picture of the second order and the fractional order (i.e. nonlocal) theories.  We adopt this view in the current work, and take this section to discuss explicitly how our result relates to the relevant second order theory. 

\subsection{Recovery of A Second Order Envelope as $\sig\to2$}
One may expect that as $\sig \to 2$ the envelopes $\Gamma_\sig$ should
behave more and more like the convex envelope of $u$. This is almost the case at least whenever $u$ is
$C^{1,1}$ from above, the discrepancy arises only because of the behavior of $E_\sig$ as $\sig \to 2$ (see Section \ref{sec:Envelope}).

\begin{prop}\label{prop:UniformConvergenceEnvelopes}
Assume $u$ is $C^{1,1}$ from above, then as $\sig\to 2$ the envelopes
$\Gamma_\sigma$
converge uniformly to a function $\Gamma_2$ which solves the obstacle problem
\begin{equation*}
\min \left \{ u(x)-\Gamma_2(x), L(\Gamma_2,x) \right \} = 0\;\; \forall\;x \in B_3,\;\; \Gamma_2(x) = 0 \;\; \forall\; x \in \partial B_3.
\end{equation*}
where $L(v,x)$ is the second order fully non-linear elliptic operator defined by
\begin{equation*}
L(v,x):= \lambda_1(D^2v(x))+\tfrac{1}{2}\Delta v(x).	
\end{equation*}
In particular, $\Gamma_0$ lies in between $u$ and the convex envelope of $u$ in $B_3$.
\end{prop}

\begin{proof}[Proof of Proposition \ref{prop:UniformConvergenceEnvelopes}]
Here we will use the operator, $E_\sig^*$, which was introduced in Remark \ref{rem:EsigStar}.  Let $\Gamma_\sig$ and $\Gamma_\sig^*$ be respectively the envelopes of $u$ made using Definition \ref{def:GammaU} with respectively the operators $E_\sig$ and $E_\sig^*$.  We note here that because $E_\sig^*\Gamma_\sig^*\geq0$ in $B_3$, then it also holds that $E_\sig\Gamma_\sig^*\geq0$ in $B_3$ as well.  Moreover, $\Gamma_\sig^*$ satisfies $\Gamma_\sig^*\leq u$ by construction.  Hence because $\Gamma_\sig$ is the supremum of such functions, we see that
\begin{equation*}
\Gamma_\sig^*\leq \Gamma_\sig.
\end{equation*}

\
Due to the $C^{1,1}$ assumption on $u$ we can use Lemma
\ref{lem:GammaRegularityByAbove} to conclude that both families, $\{\Gamma_\sig\}_\sig$ and $\{\Gamma_\sig^*\}_\sig$ are
uniformly $C^{1,1}$ from above. In particular there exists a constant
$C>0$ independent of $\sigma$ such that
\begin{equation*}
-\Laplace{\sig/2}\Gamma_\sig \leq C\ \ \text{and}\ \ -\Laplace{\sig/2}\Gamma_\sig^* \leq C
\end{equation*}
Since we also have by construction both 
\begin{equation*}
-\Laplace{\sig/2}\Gamma_\sig \geq 0\ \ \text{and}\ \ -\Laplace{\sig/2}\Gamma_\sig^* \geq 0,
\end{equation*}
we conclude that  both $\|\Laplace{\sig/2}\Gamma_\sig \|_{L^\infty(B_3)}\leq
C$ and $\|\Laplace{\sig/2}\Gamma_\sig^* \|_{L^\infty(B_3)}\leq
C$.   Since each $\Gamma_\sig$ and $\Gamma_\sig^*$ are identically zero outside $B_3$ we may
use (for instance) the Poisson kernel for $\Laplace{\sig}$ in
\cite{Land-72} to conclude both $\Gamma_\sig$ and $\Gamma_\sig^*$ are H\"older continuous in
$\real^n$ and uniformly in $\sig$. 

Therefore, from each sequence
$\sig_k \to 2$ we may select a subsequence that converges uniformly in
$\real^n$ to some functions $\Gamma_2$ and $\Gamma_2^*$.  

We note, for any smooth $v$ with enough decay at infinity, the convergence of $E_\sig(v,x)$ as $\sig \to 2$ to $\lambda_1(D^2v(x))+\frac{1}{2}\Delta v(x)$ (see Remark \ref{rem:DerivativesCompactFormula} and Proposition \ref{prop:LimitSigmaTo2HessionGamma}) and the convergence of $E_\sig^*(v,x)$ to $\lam_1(D^2v(x))$.  Thus, by the stability of viscosity solutions under uniform limits, we conclude $\Gamma_2$ and $\Gamma_2^*$ are the unique viscosity  solutions to their respective limiting obstacle problems, and hence all subsequential limits converge to the same function.  But in the limit $\Gamma_2^*$ is the convex envelope of $u$, which concludes the proposition.
\end{proof}

We proved the proposition assuming $u$ is $C^{1,1}$ from above for the sake of simplicity, but as it is reasonable to expect the envelopes $\Gamma_s$ to be somewhat regular regardless of $u$, it is likely this convergence holds under much more general circumstances.

\subsection{Derivatives of $P$ Converge To Derivatives of $v$}
We now revisit the peculiar form of (\ref{eq:Pxixj}) and (\ref{eq:TraceD2POnGamma}) mentioned in Remark \ref{rem:DsigmaPeculiar}.  We give further justification of formula (\ref{eq:Pxixj}) by checking
with a direct computation that when $P:= (-\Delta)^{-(2-\sig)/2}v$ then as $\sigma \to 2^-$ it gives back the
Hessian of $v$, assuming $v$ is smooth enough. This is not surprising since $P$, being defined as $(-\Delta)^{-(2-\sig)/2}v$ converges to $v$ uniformly as
$\sigma \to 2^-$, so if $v$ is smooth enough we can already conclude $D^2P \to D^2v$, in this regard the proposition below is only complementary.

\begin{prop}\label{prop:LimitSigmaTo2HessionGamma}
If $v$ is a fixed function which is $C^{1,1}$ from above in $B_3$ and $P$ is its Riesz potential (\ref{eq:Pressure}), then
\begin{equation*}
\lim \limits_{\sigma \to 2^-}   D^2P(x) = D^2v(x), \;\;\forall\;x\in B_3.
\end{equation*}
\end{prop}

\begin{proof}[Proof of Proposition \ref{prop:LimitSigmaTo2HessionGamma}]
The proof has two broad steps.

Step 1) We shall compute a formula for $\lim
\limits_{\sig\to2^-}P_{x_ix_j}$. To do this, we estimate the integral on
the right hand side of (\ref{eq:Pxixj}) by breaking it in two parts.
Fix $\epsilon>0$, observe that
\begin{equation*}
\left |(2-\sigma) \int_{B_\epsilon^C}\del
v(x,y)\left[\frac{y_iy_j}{\abs{y}^{n+\sig+2}}-\frac{\del_{ij}}{(n+\sig)\abs{y}^{n+\sig}}\right]dy  \right
|
\end{equation*}
\begin{equation*}
\leq \|v\|_\infty (2-\sigma) C_n
\int_{B_\epsilon^C}\frac{1}{|y|^{n+\sigma}}dy=\|v\|_\infty(2-\sigma)C_n\int_\epsilon^{+\infty}\frac{\omega_n}{t^{1+\sigma}}dt
\end{equation*}

The right hand side vanishes as $\sigma \to 2$ (for $\epsilon>0$ fixed), thus
\begin{equation}\label{eq:LimitPxixjA}
\lim \limits_{\sig\to 2^-}P_{xixj}=     \frac{n(n+2)}{2}\lim
\limits_{\sig\to 2^-}\int_{B_\epsilon(0)} A(n,2-\sig) \frac{\del
v(x,y)}{|y|^{n+\sigma}}
\left[\frac{y_iy_j}{\abs{y}^2}-\frac{\del_{ij}}{(n+\sig)}\right]dy
\end{equation}
This holds for all $\epsilon>0$ (which makes clear we are getting a
local operator). To estimate the integral inside $B_\epsilon$ recall
that $v$ is $C^{1,1}(x)$, which means that $\nabla v(x)$ and
$D^2v(x)$ exist in the sense that as $|y|\to 0$
\begin{equation*}
v(x+y)=v(x)+\nabla v (x) \cdot y +\frac{1}{2}y^T
(D^2v(x))y + o(|y|^2)
\end{equation*}
In particular, setting $\hat{y}=y/\abs{y}$
\begin{equation*}
\frac{\del v(x,y)}{|y|^2}= \hat{y}^T(D^2v(x))\hat{y}+o(1)
\;\;,\; \abs{y}\to 0
\end{equation*}
Using this expansion in (\ref{eq:LimitPxixjA}) we get
further\footnote{also recall that $A(n,2-\sig)/(2-\sig) \to
\omega_n^{-1}$ as $\sig\to2^-$}
\begin{equation*}
\lim\limits_{\sigma\to2^-} P_{x_ix_j} = \frac{n(n+2)}{2\omega_n}\lim
\limits_{\sig\to2^-} \left
\{(2-\sig)\int_{B_\epsilon}\frac{\hat{y}^T(D^2v(x))\hat{y}}{\abs{y}^{n-2+\sig}}\left[
\hat{y}_i\hat{y}_j -\frac{\del_{ij}}{n+\sig}\right]dy\right \}
\end{equation*}
We may also use polar coordinates to see that
\begin{equation*}
\int_{B_\epsilon}\frac{\hat{y}^T(D^2v(x))\hat{y}}{\abs{y}^{n-2+\sig}}\left[
\hat{y}_i\hat{y}_j -\frac{\del_{ij}}{n+\sig}\right]dy=
       \int_0^\epsilon \int_{S^{n-1}}e^T(D^2v(x))e\left[ e_ie_j
-\frac{\del_{ij}}{n+\sig}\right]\frac{t^{n-1}}{t^{n-2+\sigma}}dS(e)dt
\end{equation*}
\begin{equation*}
=\frac{1}{2-\sig}\epsilon^{2-\sig}\int_{S^{n-1}}e^T(D^2v(x))e\left[
e_ie_j -\frac{\del_{ij}}{n+\sig}\right]dS(e)
\end{equation*}
This gives us the formula
\begin{equation}\label{eq:LimitPxixjB}
\lim \limits_{\sigma \to 2^-} P_{x_ix_j} =
\frac{n(n+2)}{2\omega_n}\int_{S^{n-1}}\left
(e_ie_j-\frac{1}{n+2}\delta_{ij} \right )\; e^T(D^2v(x))e\;dS(e)
\end{equation}
Step 2) All there is left to show is that the expression on the right
in (\ref{eq:LimitPxixjB}) always gives back the $ij$ entry of
$D^2v(x)$. By rotation invariance we may assume without loss of
generality that $D^2v(x)$ is diagonal, in which case we have
$e^T(D^2v(x))e= \sum \limits_{l=1}^n v_{ll}(x)e_l^2$ for all
$e$, in other words:
\begin{equation*}
\lim \limits_{\sigma \to 2^-} P_{x_ix_j} =
\frac{n(n+2)}{2\omega_n} \	int_{S^{n-1}}\left
(e_ie_j-\frac{\delta_{ij}}{n+2} \right )e_l^2\;dS(e)v_{x_lx_l}(x)
\end{equation*}
Note that by reflection symmetry the integral of $e_ie_je_l^2$ must be
zero for all $l$ when $i\neq j$, given that $D^2v(x)$ is diagonal
this means  $\lim \limits_{\sig\to2^-}P_{x_ix_j}=0=v_{x_ix_j}(x)$
for $i \neq j$. We are left to consider the case $i=j$, to fix ideas
we do it for $i=1$. Then
\begin{equation*}
\lim \limits_{\sigma \to 2^-} P_{x_1x_1} =
\frac{n(n+2)}{2\omega_n} \sum\limits_{l=1}^n \int_{S^{n-1}}\left
(e_1^2-\frac{1}{n+2} \right )e_l^2\;dS(e)v_{x_lx_l}(x)
\end{equation*}
Due to rotation invariance this can be rewritten as
\begin{equation*}
\frac{n(n+2)}{2\omega_n}\left [ \int_{S^{n-1}}\left
(e_1^4-\frac{e_1^2}{n+2} \right )\;dS(e)v_{x_1x_1}(x)+\left (
\int_{S^{n-1}} \left (e_1^2e_2^2-\frac{e_2^2}{n+2} \right)\;dS(e)
\right ) \sum\limits_{l=2}^n v_{x_lx_l}(x)\right ]
\end{equation*}
One may compute explicitly the integrals above and get
\begin{equation*}
\int_{S^{n-1}}e_i^2dS(y)=\frac{\omega_n}{n},\;\;
\int_{S^{n-1}}e_i^2e_j^2dS(y) = \left \{ \begin{array}{rl}
\frac{\omega_n}{n(n+2)} & i\neq j\\
\\
\frac{3\omega_n}{n(n+2)} & i =j
\end{array}\right.
\end{equation*}
Therefore
\begin{equation*}
\int_{S^{n-1}}\left (e_1^4-\frac{e_1^2}{n+2} \right
)\;dS(e)=\frac{2\omega_n}{n(n+2)}, \;\;	 \int_{S^{n-1}} \left
(e_1^2e_2^2-\frac{e_2^2}{n+2} \right)\;dS(e) = 0
\end{equation*}
Plugging this in the last expression for $\lim
\limits_{\sig\to2^-}P_{x_1x_1}(x)$ we obtain
\begin{equation*}
\lim \limits_{\sig\to2^-}P_{x_1x_1}(x)=\frac{n(n+2)}{2\omega_n}\left (
\frac{2\omega_n}{n(n+2)} v_{x_1x_1}+0\right ) = v_{x_1x_1}
\end{equation*}
This proves the proposition.
\end{proof}

\begin{rem} 
The spherical integrals above are not hard to compute,
they follow from counting the different terms appearing in the trivial
identities:
\begin{equation*}
\int_{S^{n-1}}e_1^2+...+e_n^2\;dS(e)=\int_{S^{n-1}}\left (
e_1^2+...+e_n^2\right )^2dS(e)=\int_{S^{n-1}}dS(e)=\omega_n
\end{equation*}
and from the relation
$\int_{S^{n-1}}e_1^2e_2^2dS(e)=\frac{1}{3}\int_{S^{n-1}}e_1^4dS(e)$
which is a standard computation for $n=2$ and can be pushed for all
$n$ via  induction (integrating along slices of the sphere and
rescaling the lower dimensional formula in the inductive step).
\end{rem}

\subsection{Operators covered in the limit $\sig\to2^-$} We now characterize the second order operators that can be obtained as a limit of fractional order operators of the form~\eqref{eq:LinearGUISCH}. It is rather surprising that many elliptic linear operators of second order cannot be obtained as limits of operators of the form~\eqref{eq:LinearGUISCH}, as opposed to what was expected (see discussion in \cite[Section 3]{CaSi-09RegularityIntegroDiff}).

\begin{prop}
Given a symmetric $n\times n$ matrix $M$ and $v \in C^2_{loc}(\real^n) \cap L^\infty(\real^n)$ we define
\begin{equation*}
L_{M,\sig}(v,x)=\Tr( M \cdot h_\sig(v,x)) = \frac{(n+\sig-2)(n+\sig)}{2}A(n,2-\sig)\int_{\real^n}\delta v(x,y)\frac{y^T M y}{|y|^{n+\sig+2}}dy.
\end{equation*}
Then, for any $x \in \real^n$ we have
\begin{equation*}
\lim \limits_{\sig\to2^-} L_{M,\sig}(v,x) = \Tr\left ( (M -\tfrac{\Tr(M)}{n+2}\Id ) \cdot D^2v(x) \right ).	
\end{equation*}
Moreover, an operator of the form $L(v,x)=\Tr(A\cdot D^2v(x))$ can be obtained as a limit of the operators $L_{M,\sig}$ for some $M\geq 0$ if and only if,
\begin{equation}\label{eqn:LimitMatrixProperty}
A-\tfrac{\Tr(A)}{n+2}\Id \geq 0.	
\end{equation}
\end{prop}
\begin{proof}
This proposition is merely and application of Lemma \ref{lem:SingularIntegralPij}, which in our case says that, if we consider the function $P^\sig_v=(-\Delta)^{-(\tfrac{2-\sig}{2})}v$ then
\begin{align*}
\Delta P(x) & = \tfrac{(n+\sig-2)(n+\sig)}{2}A(n,2-\sig)\int_{\real^n}\frac{\delta \Gamma(x,y)}{|y|^{n+\sig}}\left (1-\tfrac{n}{n+\sig} \right )dy	\\
& = \tfrac{(n+\sig-2)\sig}{2}A(n,2-\sig)\int_{\real^n}\frac{\delta \Gamma(x,y)}{|y|^{n+\sig}}dy.\\
\end{align*}
In particular, we get
\begin{equation*}
-\Id\tfrac{(n+\sig-2)}{2}A(n,2-\sig)\int_{\real^n}\frac{\delta \Gamma(x,y)}{|y|^{n+\sig}}dy	=-\tfrac{1}{\sig}\Delta P(x).
\end{equation*}
Then, again by Lemma \ref{lem:SingularIntegralPij} and the definition of $h^\sig(v,x)$, we have
\begin{equation*}
h_\sig(v,x)= D^2P(x)+\tfrac{\Delta P(x)}{\sig}\Id.
\end{equation*}
Thus,
\begin{align*}
L_{M,\sig}(v,x) & = \Tr\left (M\cdot (D^2P^\sig_v(x)+\tfrac{\Tr(D^2P^\sig_v(x))}{\sig}\Id ) \right ) \\
& = \Tr(M\cdot D^2P^\sig_v(x))+\tfrac{1}{\sig}\Tr(M)\Tr(D^2P^\sig_v(x	))\\
& = \Tr \left ( (M +\tfrac{\Tr(M)}{\sig}\Id ) \cdot D^2P^\sig_v(x) \right ).	
\end{align*}
Since $v \in C^2_{loc}(\real^n)$ we have $D^2P^\sig_v \to D^2v(x)$ for every $x$ as $\sig\to 2^-$, and
\begin{equation*}
\lim \limits_{\sig \to 2^-}L_{M,\sig}(v,x) = \Tr(A\cdot D^2v(x)),\;\;A:=M+\tfrac{\Tr(M)}{2}\Id. 	
\end{equation*}
Now, we invert the relation defining $A$ in terms of $M$, it is worth noticing that this relation is nothing else but $A=M_\sig$ when $\sig=2$, as defined previously in \eqref{eq:ASigma}. Taking the trace on both sides gives us $\Tr(A)=(1+\tfrac{n}{2})\Tr(M)$, so that $M=A-\tfrac{\Tr(M)}{2}\Id=A-\tfrac{\Tr(A)}{n+2}\Id$, from where it follows that $M\geq 0$ if and only if $A-\tfrac{\Tr(A)}{n+2}\Id \geq 0$, and the proposition is proved.
\end{proof}
\begin{rem}
Note that for $A=\lambda \Id$ ($\lambda > 0)$ we have $A-\tfrac{\Tr(A)}{n+2}\Id= \lambda \tfrac{2}{n+2}\Id \geq 0$. In particular, in the space of matrices there is a tubular neighborhood of this ray on which \eqref{eqn:LimitMatrixProperty} holds. 
\end{rem}
\begin{rem}
As it was natural to expect, the operator $L(v,x)$ appearing in Proposition \ref{prop:LimitSigmaTo2HessionGamma} (once linearized about some $v$) satisfies \eqref{eqn:LimitMatrixProperty}.	
\end{rem}

\section{Comparison Theorems Related to Theorem \ref{thm:Main}}\label{sec:Useful Consequences}
\setcounter{equation}{0}

In this section we collect various results that are either direct applications of Theorem \ref{thm:Main} or straightforward modifications of its proof.  Each of the results here are stated without proof, and we simply mention some of the modifications.  First we mention in Theorem \ref{thm:GeneralDomains} the analog of Theorem \ref{thm:Main} to more general domains than just $B_1$, and in Theorem \ref{thm:Comparison} the applications to the comparison principle for (\ref{eq:PIDEgeneral}).

\begin{thm}\label{thm:GeneralDomains}
Let equation (\ref{eq:PIDEmain}) be set in a general bounded, connected, domain, $D$ instead of in $B_1$.  Then 
\begin{equation*}
-\inf_{D}\{u\} \leq \frac{C(n)}{\lam}\diam(D)(\norm{f}_{L^\infty(K_u)})^{(2-\sig)/2}(\norm{f}_{L^n(K_u)})^{\sig/2}.
\end{equation*}
\end{thm}

We note that to modify the proof of Theorem \ref{thm:Main} to incorporate $D$, one simply needs to modify the domain of truncation in the definition of $\Gamma^\sig_u$, Definition \ref{def:GammaU}, and also the selection of the radius, $R_\sig$, from Lemma \ref{lem:PotenialOscillation}.  In particular, if we define the set $D_3:=\{x: distance(x,D)\leq 3\}$, then we are concerned with the obstacle problem in $D_3$, using $u\Indicator_{D}$ in Definition \ref{def:GammaU}.  Furthermore, the ball $B_3$ is no longer used in the proof of Lemma \ref{lem:PotenialOscillation}, but instead $R_\sig$ is chosen large enough to compensate for the size of $D_3$ instead of the size of $B_3$.

An immediate consequence of the fact that the difference between a subsolution and a supersolution with solve the minimal equation for $M^-$ is that we get a comparison theorem for subsolutions and supersolutions of equations with the same operator, $F$, but different right hand sides.

\begin{thm}\label{thm:Comparison}
Suppose that $F$ is in the elliptic family for $M^-$ and $M^+$.
Let $u$ and $v$ be bounded and respectively usc subsolution and lsc supersolution of
\begin{equation*}
\begin{cases}
\ds F(u,x)\geq f(x) &\text{ in } D\\
u= u_0 &\text{ on } \real^n\setminus D
\end{cases}
\end{equation*}
and 
\begin{equation*}
\begin{cases}
\ds F(v,x)\leq g(x) &\text{ in } D\\
v=v_0 &\text{ on } \real^n\setminus D.
\end{cases}
\end{equation*}
then 
\begin{align*}
&\sup_{D}\{u-v\}\\
&\leq \sup_{\real^n\setminus D}\{u_0-v_0\}+\frac{C(n)}{\lam}\diam(D)(\norm{(f-g)^-}_{L^\infty(K_{u-v})})^{(2-\sig)/2}(\norm{(f-g)^-}_{L^n(K_{u-v})})^{\sig/2}.
\end{align*}
\end{thm}

\section{Special Cases of the Regularity Theory of Caffarelli and Silvestre}\label{sec:NonlocalLepsilon}
\setcounter{equation}{0}

In this section we show how Theorem \ref{thm:Main} can be used to
prove the standard $L^\epsilon$ estimates for viscosity solutions of
(\ref{eq:PIDEmain}) (at least for a special family of operators). None
of these results are new, as such estimates have already been proved
in \cite[Section 10]{CaSi-09RegularityIntegroDiff} for very general nonlinear
integro-differential equations. The $L^\epsilon$ estimate constitutes
the backbone of the regularity theory for fully nonlinear elliptic
equations; it is the key fact behind the Harnack inequality of
Krylov-Safonov, the Evans-Krylov theorem and the respective
Caffarelli-Silvestre theorems for nonlocal equations (again, see
\cite{CaCa-95} and \cite{CaSi-09RegularityIntegroDiff}). Our purpose in revisiting this
part of the theory in our case is basically illustrating the uses of
Theorem \ref{thm:Main}, in particular Theorem \ref{thm:Main} allows us
to do essentially the same proof of the $L^\epsilon$ bound for second
order equations used in \cite[Lemma 4.6]{CaCa-95} (compare to the different method needed in \cite[Section 10]{CaSi-09RegularityIntegroDiff}).
Furthermore, the estimates obtained are uniform in the order of the equation,
recovering the second order theory as $\sigma \to 2$, which was already done
in \cite{CaSi-09RegularityIntegroDiff}

We state without proof the following proposition regarding the construction of a
special barrier function. It follows by an argument similar to that used in Lemma 9.1 of \cite{CaSi-09RegularityIntegroDiff}.

\begin{prop}\label{prop:SpecialFunction}
Given $0<\lambda \leq \Lambda$ and $\sigma_0 \in (0,2)$ there exist
constants $C_0,M>0$ and a $C^{1,1}$ function $\eta(x):\mathbb{R}^n
\to\mathbb{R}$ such that
\begin{enumerate}
       \item $\mbox{supp} \; \eta \subset B_{2\sqrt{n}}(0)$
       \item $\eta \leq -2$ in $Q_3$ and $\|\eta\|_\infty \leq M$
       \item For every $\sigma>\sigma_0$ we have $M^+(\eta,x)   \leq C_0
\xi$ everywhere where $\xi$ is a continuous function with support
inside $B_{1/4}(0)$ and such that $0\leq \xi\leq 1$.
\end{enumerate}
       \end{prop}
With this special function in hand, Theorem \ref{thm:Main} allows
(just as in the second order theory) to control the \emph{average
size} of a supersolution in terms of its \emph{value at a point}.
This point to average estimate implies a weak $L^\epsilon$ estimate
for supersolutions which is the key step in the proof of the
Krylov-Safonov and Caffarelli-Silvestre regularity theorems.

\begin{lem}\label{lem:LEpsilonBaseCase}
Given $n,\lambda,\Lambda$ and $\sigma_0$ such that $0<\lambda\leq
\Lambda$ and $\sigma_0 \in (0,2)$, one can find positive constants
$M>1,\mu<1$ and $\delta_0$ such that if $u$ satisfies:
\begin{enumerate}
       \item $u \geq 0$ in $\mathbb{R}^n$
       \item $\inf \limits_{Q_3} u \leq 1$
       \item $M^-(u,x) \leq f$ in $Q_{4\sqrt{n}}$ (for some
$\sigma>\sigma_0$), and $\|f\|_{L^n(Q_{4\sqrt{n}})} \leq \delta _0$,
$\|f\|_{L^\infty(Q_{4\sqrt{n}})} \leq 1$.
\end{enumerate}
Then we have the bound
\begin{equation}\label{eq:LEpsilon1}
\left | \left \{ x\in Q_1: u(x)\leq M \right \} \right | \geq \mu |Q_1|.
\end{equation}

\end{lem}

\begin{proof}[Proof of Lemma \ref{lem:LEpsilonBaseCase}]
Consider the function $w=u+\eta$ ($\eta$ as in Proposition
\ref{prop:SpecialFunction}), then $w$ satisfies (in the viscosity
sense)
\begin{equation}\label{eq:LEpsilonMain}
\begin{cases}
\ds M^-(w,x)\leq f(x)+C_0\xi &\text{ in } B_{4\sqrt{n}}\\
w\geq 0 &\text{ on } \real^n\setminus B_{4\sqrt{n}}.
\end{cases}
\end{equation}
Moreover $-\inf \limits_{Q_3}\{w\} \geq 1$ since $u\leq 1$ somewhere
in $Q_3$ and $\eta \leq -2$ everywhere in $Q_3$. In this situation,
Theorem \ref{thm:Main} (rescaled to the ball $B_{4\sqrt{n}}$) tell us
that
\begin{equation*}
1 \leq C \lam^{-n}\left (1+C_0\right
)^{(2-\sig)/2}(\norm{f+C_0\xi}_{L^n(K_w)})^{\sig/2},
\end{equation*}
where we recall $K_\omega = \{ x \in B_{4\sqrt{n}}: w(x)=\Gamma_w(x)\}$, then
\begin{equation*}
C^{-2/\sigma}\lam^{2n/\sigma}(1+C_0)^{-(2-\sigma)/\sigma}\leq
\norm{f+C_0\xi}_{L^n(K_w)}.
\end{equation*}
Hence
\begin{equation}
 C^{-2/\sigma}\lam^{2n/\sigma}(1+C_0)^{-(2-\sigma)/\sigma}\leq
\delta_0+C_0|K_w \cap B_{1/4}|^{1/n}.
\end{equation}
One then sees that picking $\delta_0$ universally small one gets for a
universal $\mu \in (0,1)$ the lower bound
\begin{equation}
\mu |Q_1| \leq |K_w \cap B_{1/4}| \leq |K_w \cap Q_1|.
\end{equation}
Now $x \in K_w$ implies in particular that $w \leq 0$ therefore
$u\leq-\eta \leq M$, where $M=\sup |\eta|$. Since $K_w \cap Q_1
\subset \{x\in Q_1: u \leq M \} $ this last inequality proves the
lemma.
\end{proof}

As done in \cite[Chapter 4, Section 4.2]{CaCa-95} one can use the
Calder\'on-Zygmund decomposition and \ref{lem:LEpsilonBaseCase} to
prove the weak $L^\epsilon$ estimates we previously mentioned.

\begin{thm}[weak $L^\epsilon$ estimate]\label{thm:LEpsilonEstimate}
Let $u \geq 0$ in $\mathbb{R}^n$ be a supersolution of $M^-(u,x)\leq
f$ in $B_1$ (for $\sigma>\sigma_0$) and such that $u(0) \leq 1$.
Suppose that
\begin{equation*}
\|f\|_\infty \leq 1,\;\; \|f\|_{L^n(B_2)}\leq\delta_0,\;\;\sigma \in
(\sigma_0,2).
\end{equation*}
Then there are universal constants $C,\delta_0$ and $\epsilon$ (i.e.
determined by $n,\sigma_0,\lambda$ and $\Lambda$) such that for all
$t>0$ we have
\begin{equation}
\left | \left \{ u>t \right \} \cap B_{1/2} \right | \leq Ct^{-\epsilon}.
\end{equation}

\end{thm}

With this lemma in hand, that viscosity solutions to
(\ref{eq:PIDEmain}) are H\"older continuous or that they satisfy a
Harnack inequality follows by standard arguments that can be found in
\cite[Theorem 4.3 and Proposition 4.10]{CaCa-95} for second order
equations and in \cite[Theorem 12.1 and Theorem 11.1]{CaSi-09RegularityIntegroDiff} for
integro-differential equations.
\section{Applications and Open Problems}\label{sec:Applications}
\setcounter{equation}{0}

Here we would like to make a few remarks about further research
directions and open questions where we anticipate Theorem
\ref{thm:Main} could be useful, some of which are in obvious analogy
to the second order theory. Accordingly, the discussion below is only suggestive but we include it with the hope of
stimulating further work.\\

\noindent
\textbf{Stochastic homogenization}. As pointed out in the
introduction, Theorem \ref{thm:Main} will play an important role in
the homogenization of stationary ergodic families of equations within
the ellipticity class governed by $M^-$, and this will be presented in
\cite{Schw-11Stoch}.\\

\noindent
\textbf{``$W^{\sig,\ep}$'' and ``$W^{\sig,p}$'' estimates.}  These would be analogous to the $W^{2,\ep}$ estimates of Caffarelli \cite[Chapter 7]{CaCa-95} and Lin \cite{Lin-1986SecondDerviLpEstimates} and subsequently the
$W^{2,p}$ theory of Caffarelli \cite{Caff-1989InteriorEstimates}.  Such estimates
would allow use of the regularity theory of
\cite{CaSi-09RegularityIntegroDiff,CaSi-09RegularityByApproximation,CaSi-09EvansKrylov} to
obtain $W^{\sig,p}$ regularity for viscosity solutions of
(\ref{eq:PIDEgeneral}) in terms of the $L^p$ norm of the right hand
side $f$. It is important to remark that such estimates are not yet
available for any kind of fully nonlinear equation of fractional
order.\\

\noindent
\textbf{More general equations}. Comparison and regularity theory for
more general equations where their ellipticity is considered on
choices of $L_A$ (with varying scopes of generality) such as:
\begin{equation}
L_A(v,x) =(2-\sig)\int_{\real^n}\del u(x,y)a(x,y)\abs{y}^{-n-\sig}dy
\end{equation}
where $a(x,y)$ is homogeneous of degree $0$ in $y$,
\begin{equation}
L_A(v,x) =(2-\sig)\int_{\real^n}\del u(x,y)(y^TA(x)y) n(dy)
\end{equation}
where $n(dy)$ is comparable to $\abs{y}^{-n-\sig}dy$ only \emph{in
some subset} of $\real^n$ (indicated to the authors in \cite{Kass-2010PersonalComm}),
\begin{equation}
L_A(v,x) =(2-\sig)\int_{\real^n}\del u(x,y)\frac{a(x,y)}{\abs{y}^{n+\sig}}dy,\;\;\; \lam \leq a \leq \Lam
\end{equation}
which corresponds to the family in \cite{CaSi-09RegularityIntegroDiff}, and
\begin{equation}
L_A(v,x) =(2-\sig)\int_{\real^n}\del u(x,y)n(x,dy)
\end{equation}
where $n(x,dy)$ (for each $x$) is again a measure comparable to
$\abs{y}^{-n-\sig}dy$ only in some subsets of $\real^n$ (for example,
along certain directions through the origin).\\

\noindent
\textbf{The ``right'' nonlocal Monge-Amp\'ere equation.} It is
worthwhile to find out whether another notion of nonlocal determinant
can be found which is both extremal and carries a geometric
interpretation (or ``divergence structure'') such that it achieves the
key features listed in Section
\ref{sec:Sketch}--(\ref{eq:SketchIntegralDsigma}),
(\ref{eq:SketchDsigmaMminus}), and
(\ref{eq:SketchDsigmaOffContactSet})-- for a more general family of
equations,  such as those considered in \cite{CaSi-09RegularityIntegroDiff}.\\


\section{Appendix-- Concave Nonlocal Obstacle Problems}\label{sec:Appendix}
\setcounter{equation}{0}

In this section, we develop some regularity results for solutions to obstacle problems with concave nonlocal operators.  We believe these results may be of independent interest and have hence chosen to present them in a generic form, but the example we actually use in this work is for the operator, $E_\sig$.  The results in this section are nonlocal analogs to the results of \cite{BrKi-1975RegNonLinVarIneq} and \cite{LeSt-1971SmoothnessNonCoerciveVarIneq}; the interested reader can also find them presented in the book \cite[Chapter IV]{KiSt-1980BookIntroVarIneq}.

First we list some notation and assumptions for this section:
\begin{enumerate}
\item $\psi:B_3\to\real$ is the obstacle function with $\psi\leq0\in B_3$ and $\psi\geq0\in \real^n\setminus B_3$.
\item $F$ is a degenerate elliptic, \textbf{translation invariant}, nonlocal operator which is concave with respect to $v$ in the sense that there is a collection of linear $\sig$-order operators, $\L$, and $F$ can be represented as
\begin{equation}\label{eq:AppendixConcave}
F(v,x) = \inf_{L\in\L}\{L(v,x)\}.
\end{equation}
\item $F$ satisfies a non-degeneracy type condition as
\begin{equation}\label{eq:AppendixFBelowLaplace}
-\Laplace{\sig/2}v(x)\geq F(v,x).	
\end{equation}
\item $\ds \beta_\ep(s)=\beta(\frac{s}{\ep})$ where $\beta:\real\to\real$ is convex, non-negative, decreasing, $\beta(0)\geq \sup_{x\in B_3}F(\psi,x)$, and $\beta(s)=0$ for $s\geq1$.
\item $\Omega\subset\subset B_3$.
\end{enumerate}

The main function of interest for this section is the solution to the obstacle problem, $v$, solving
\begin{equation}\label{eq:AppendixObstacleMain}
v=\sup\{w\ :\ F(w,x)\geq0\ \text{in}\ B_3\ \text{and}\ w\leq \psi\Indicator_{B_3}\ \text{in}\ \real^n\}.
\end{equation}
\noindent
It is standard that $v$ is the unique viscosity solution of the equation

\begin{equation}\label{eq:AppendixObstacleViscosityEq}
\begin{cases}
\ds \min\{F(v,x),u-v\} = 0 &\text{ in } B_3\\
v=0 &\text{ on } \real^n\setminus B_3.
\end{cases}
\end{equation}

\noindent
One of the main tools for investigating (\ref{eq:AppendixObstacleMain}) is the penalized equation, which also admits a unique viscosity solution:
\begin{equation}\label{eq:AppendixPenalized}
\begin{cases}
\ds F(v^\ep,x) = \beta_\ep(\psi-v^\ep) &\text{ in } B_3\\
v^\ep=0 &\text{ on } \real^n\setminus B_3.
\end{cases}
\end{equation}

\noindent
Finally, in order to make some arguments easier by working with classical solutions, we will use the regularized, penalized equation:
\begin{equation}\label{eq:AppendixRegPenalized}
\begin{cases}
\ds F(v^{\ep\rho},x) +\rho\Delta v^{\ep\rho} = \beta_\ep(\psi-v^{\ep\rho}) &\text{ in } B_3\\
v^{\ep\rho}=0 &\text{ on } \real^n\setminus B_3.
\end{cases}
\end{equation}

The main result of this section says that the solutions of these obstacle problems not only inherit the regularity from above of the obstacle, but in fact they are even regularizing in the sense that they obey interior regularity from below as well, which may in fact be more regular than the obstacle, $\psi$.

\begin{prop}\label{prop:BrezisKinderNonlocal}	
If $\psi$ satisfies classically $-\Laplace{\sig/2}\psi(\cdot)\leq C$ and $\Om\subset\subset B_3$ then there exists a constant, $C_1$, depending only on $C$, $\sig$, $n$, $dist(\Om,\partial B_3)$, and $\norm{\psi}_{L^\infty}$ such that for a.e. $x\in\Om$,
$v$ satisfies 
\[
\int_{\real^n}\frac{\abs{\del\Gamma(x,y)}}{\abs{y}^{n+\sig}}dy <\infty
\]
and
\[
0\leq -\Laplace{\sig/2}v(x)\leq C_1\ \ \text{and}\ \ \norm{\Gamma}_{H^\sig(\Om)}\leq C_1.
\]
\end{prop}

\begin{rem}
It is worth pointing out that the same results of this section hold \textit{mutatis mutandi} if one replaces the operator $-\Laplace{\sig/2}$ by any other linear, translation invariant, $\sig$-order operator (and its corresponding kernel, $K$) which has similar regularizing properties to those of $-\Laplace{\sig/2}$ used in the proof of Lemma \ref{lem:AppendixVepCovergeToV}.  
\end{rem}

\noindent
In fact, if $\psi$ has enough regularity such that for all of the translation invariant $L$ in the ellipticity class of \cite[Section 3]{CaSi-09RegularityIntegroDiff} (the operators governed by the extremal operator of (\ref{eq:MminusDefCaSi})) it holds that $L(\psi,x)\leq C$, then a much stronger result holds.  The proof is a simple calculation, and is found in \cite[Section 7]{CaSi-09EvansKrylov}, but we include it here for completeness.

\begin{cor}\label{cor:AppendixAllOperators}
Let $\L$ be the collection of linear operators, $L$, of the form 
\[
L(w,x) = \int_{\real^n}\del w(x,y)K(y)dy\ \ \text{where}\ K\ \text{is measurable and}\ \ \frac{\lam}{\abs{y}^{n+\sig}}\leq K(y)\leq \frac{\Lam}{\abs{y}^{n+\sig}}.
\]
If $\psi$ satisfies $L(\psi,x)\leq C$ for all $L\in\L$, then there exists a constant, $C_1$, which depends only on $C$, $\sig$, $n$, $\lam$, $\Lam$, $dist(\Om,\partial B_3)$, and $\norm{\psi}_{L^\infty}$ such that for a.e. $x\in\Om\subset\subset B_3$,
\[
\int_{\real^n}\frac{\abs{\del\Gamma(x,y)}}{\abs{y}^{n+\sig}}dy \leq C_1.
\]
\end{cor}

\begin{proof}[Proof of Corollary \ref{cor:AppendixAllOperators}]
Proposition \ref{prop:BrezisKinderNonlocal} (modified for $L$), says that for all $L\in\L$, 
\[
0\leq L(v,x)\leq C_1,
\]
which in particular implies (see \cite[Section 3]{CaSi-09RegularityIntegroDiff})
\[
\int_{\real^n}\frac{\Lam\left(\del v(x,y)\right)^+-\lam\left(\del v(x,y)\right)^-}{\abs{y}^{n+\sig}}dy\leq C_1.
\]
On the other hand by (\ref{eq:AppendixFBelowLaplace}) and (\ref{eq:AppendixObstacleMain}), we have that 
\[
0\leq -\Laplace{\sig/2}v(x) = \int_{\real^n}\frac{\del v(x,y)}{\abs{y}^{n+\sig}}dy,
\]
and hence multiplying through by $-\lam$
\[
\int_{\real^n}\frac{-\lam\left(\del v(x,y)\right)^+ + \lam\left(\del v(x,y)\right)^-}{\abs{y}^{n+\sig}}dy\leq 0.
\]
Appropriately adding together these estimates gives and estimate on $(\del v(x,y))^+\abs{y}^{-n-\sig}$, which can be plugged back into the estimate $0\leq-\Laplace{\sig/2}v(x)$ to conclude the Corollary.
\end{proof}

\noindent
The proof of Proposition \ref{prop:BrezisKinderNonlocal} is broken up into multiple pieces which are worth recording in their own right as Lemmas \ref{lem:AppendixRegPenalizedEstimate} and \ref{lem:AppendixVepCovergeToV}.  Once these results are proved, Proposition follows immediately.

\begin{lem}\label{lem:AppendixRegPenalizedEstimate}
Let $s>0$ be fixed with $s<dist(\Om,B_3)$, and let the truncated kernel and operator be defined as
\begin{equation}\label{eq:AppendixTruncatedLaplace}
\tilde K(y):= \abs{y}^{-n-\sig}\Indicator_{\abs{y}\leq s}\ \ \text{and}\ \ \Tilde L(u,x) := \int_{\real^n}\del u(x,y)\tilde K(y)dy.
\end{equation}
Let $v^{\ep\rho}$ be the unique smooth solution to the regularized and penalized Dirichlet problem (\ref{eq:AppendixRegPenalized}) and $v^\ep$ the unique solution of (\ref{eq:AppendixPenalized}). Then $v^{\ep\rho}\to v^\ep$ as $\rho\to0$ locally uniformly, $v^{\ep\rho}$ satisfies the estimate
\begin{equation}\label{eq:AppendixFracLaplaceBoundRegularized}
\max_{x\in B_3}\{\tilde L(v^{\ep\rho},x)\}\leq \max_{x\in B_3}\{\tilde L(\psi,x)\},
\end{equation} 
and $v^\ep$ satisfies in the viscosity sense
\begin{equation}\label{eq:AppendixFracLaplaceBoundVEp}
\tilde L(v^{\ep},x)\leq\max_{x\in B_3}\{\tilde L(\psi,x)\}.
\end{equation}
\end{lem}

\begin{proof}[Proof of Lemma \ref{lem:AppendixRegPenalizedEstimate}]

The whole point of working with $v^{\ep\rho}$ is simply to have a classical solution to manipulate more easily.    By the concavity of $F$, as (\ref{eq:AppendixConcave}), and the linearity of $\tilde L$, the functions $w:=\tilde L(v^{\ep\rho})$ satisfy in $\Om$ the equation
\begin{equation}\label{eq:AppendixConvexityFracLaplace}
F(w,x)+\Delta w(x)\geq \tilde L(\beta_\ep(\psi-v^{\ep\rho}),x).
\end{equation}
The main point of choosing $\beta$ to be convex is the fact that this implies
\begin{equation}\label{eq:AppendixLbetaIneq}
\tilde L(\beta_\ep(\psi-v^{\ep\rho}),x)\geq -\norm{\beta_\ep}_{L^\infty} \tilde L((\psi-v^{\ep\rho}),x).
\end{equation}
To see this, we note that convexity gives
\[
\beta_\ep\left((\psi-v^{\ep\rho})(x+y)\right)+\beta_\ep\left((\psi-v^{\ep\rho})(x-y)\right)\geq 2\beta_\ep\left(\frac{1}{2}(\psi-v^{\ep\rho})(x+y)+\frac{1}{2}(\psi-v^{\ep\rho})(x-y)\right),
\]
which after subtracting $2(\psi-v^{\ep\rho})(x)$ and estimating with the Lipschitz norm of $\beta_\ep$ gives
\[
\del\beta_\ep \circ(\psi-v^{\ep\rho})(x,y)\geq -\norm{\beta'_\ep}_\infty\left(\del(\psi-v^{\ep\rho})(x,y)\right).
\]
Thus (\ref{eq:AppendixLbetaIneq}) follows.  Therefore, plugging back into (\ref{eq:AppendixConvexityFracLaplace}) gives
\begin{equation*}
F(w,x)+\Delta w(x)\geq -\norm{\beta_\ep}_{L^\infty} \tilde L((\psi-v^{\ep\rho}),x),
\end{equation*}
which can be evaluated at $x_\text{max}$ such that $w(x_\text{max})=\max\{w\}$ to give
\begin{equation*}
0\geq F(w,x_\text{max})+\Delta w(x_\text{max})\geq -\norm{\beta_\ep}_{L^\infty} \tilde L((\psi-v^{\ep\rho}),x_\text{max}).
\end{equation*}
We note the use of the simple fact that at a maximum point, both $F(w)$ and $\Delta w$ are nonpositive.
Therefore, multiplying through by $-\norm{\beta_\ep}_{L^\infty}$ gives (\ref{eq:AppendixFracLaplaceBoundRegularized}).

The fact that (\ref{eq:AppendixPenalized}) admits a unique viscosity solution, combined with the stability of viscosity solutions of (\ref{eq:AppendixFracLaplaceBoundRegularized}) immediately gives via standard techniques that $v^{\ep\rho}\to v^\ep$ as $\rho\to0$ locally uniformly (see any combination of \cite{BaIm-07}, \cite[Sections 4,5]{CaSi-09RegularityIntegroDiff}, \cite{CrIsLi-92}).  Finally, the stability of the inequality (\ref{eq:AppendixFracLaplaceBoundRegularized}) with respect to local uniform limits immediately gives (\ref{eq:AppendixFracLaplaceBoundVEp}), and we conclude the proof of the lemma.
\end{proof}
	
\begin{lem}\label{lem:AppendixVepCovergeToV}
Let $v$ be the unique solution of (\ref{eq:AppendixObstacleViscosityEq}).  There is a subsequence, still denoted as $v^\ep$, such that $v^\ep\to v$ locally uniformly in $B_3$ as well as weakly in $H^\sig$ as $\ep\to0$ and
\[
0\leq -\Laplace{\sig/2}v^\ep(x)\leq C_1,
\]
in the viscosity sense and
where $C_1$ depends only on $\max_{B_3}\{-\Laplace{\sig/2}\psi\}$, $\norm{\psi}_{L^\infty}$, $n$, $\sig$, $\Om$.
\end{lem}

\begin{proof}[Proof of Lemma \ref{lem:AppendixVepCovergeToV}]

We start with the observation that constants are subsolutions of $F=0$, and so $v\geq \inf_{B_3}\{\psi\}$.  Furthermore, by the definition of $\tilde L$ and (\ref{eq:AppendixFracLaplaceBoundVEp}), we have
\begin{align*}
	-\Laplace{\sig/2}v^\ep &= \int_{\abs{y}\leq s}\del v^\ep (x,y)K(y)dy + \int_{\abs{y}>s}\del v^\ep (x,y)K(y)dy\\
	&\leq \max_{B_3}\left\{\tilde L(\psi)\right\} +\tilde C(s,n,\sig)\norm{\psi}_{L^\infty}\\
	&\leq \max_{B_3}\left\{-\Laplace{\sig/2}\psi\right\} +C(s,n,\sig)\norm{\psi}_{L^\infty},
\end{align*}
where $C(s,n,\sig)$ is a constant depending only on $s,n,\sig$. 

This estimate combined with the assumptions (\ref{eq:AppendixFBelowLaplace}) and $\beta_\ep\geq0$ put into the equation (\ref{eq:AppendixObstacleViscosityEq}) gives that $v^\ep$ solves in the viscosity sense, and hence in view of the linearity also in the sense of distributions:
\begin{equation}\label{eq:AppendixVepLaplaceBounded}
	0\leq -\Laplace{\sig/2}v^\ep(x)\leq \max_{B_3}\left\{L(\psi,\cdot)\right\} +C(s,n,\sig)\norm{\psi}_{L^\infty} \ \ \text{in}\ \Om.
\end{equation}
Hence $v^\ep$ are uniformly bounded in $H^\sig$.  Furthermore, as in the proof of \cite[Proposition 2.9]{Silv-07} (or alternatively using the Green's Function of $\Laplace{\sig/2}$ for $B_3$, see \cite{Land-72}), the boundedness of $-\Laplace{\sig/2}v^\ep$ also gives a uniform $C^\al$ estimate for $v^\ep$.  Hence we can extract two subsequences $v^\ep\to w$, both locally uniformly and in $H^\sig$ weakly for some $w$.

What remains to show is to confirm that $w=v$.  We will show that $w$ necessarily solves (\ref{eq:AppendixObstacleViscosityEq}), and hence be uniqueness, $w=v$.  We first note that since $\beta_\ep\geq0$, $v^\ep$ and hence by stability of viscosity solutions, also $w$, is a subsolution of $F(w,x)\geq0$.  Furthermore, $\beta$ is chosen specifically so that $F(\psi,x)\leq \beta_\ep(0)$ and hence $\psi$ is a supersolution of (\ref{eq:AppendixPenalized}).  Thus by comparison with $v^\ep$, $v^\ep\leq\psi$ and hence $w\leq\psi$ as well.  Therefore, $w$ is a subsolution of (\ref{eq:AppendixObstacleViscosityEq}).  Now to justify the supersolution property suppose that $w-\phi$ has a strict global max at $x_0\in B_3$.  Then there exist $x_\ep\in B_3$ with $x_\ep\to x_0$ and $v^\ep-\phi$ has a local max at $x_\ep$.  If it happens that $w(x_0)=\psi(x_0)$, then (\ref{eq:AppendixObstacleViscosityEq}) is satisfied, and so we assume that $w(x_0)<\psi(x_0)$.  In this case for $\ep$ small enough, $v^\ep(x_\ep)<\psi(x_\ep)$, and hence for $\ep$ small enough, $\beta_\ep(\psi(x_\ep)-v^\ep(x_\ep))=0$.  Thus $F(\phi,x_\ep)\leq0$ by the supersolution property of $v^\ep$.  Passing to the limit using the continuity of $F$ and the smoothness of $\phi$ gives $F(\phi,x_0)\leq0$.  Hence $w$ solves (\ref{eq:AppendixObstacleViscosityEq}).  This concludes the proof of the lemma. 
\end{proof}

To conclude this section, we give the brief proof of Proposition \ref{prop:BrezisKinderNonlocal}.

\begin{proof}[Proof of Proposition \ref{prop:BrezisKinderNonlocal}]
The inequality (\ref{eq:AppendixVepLaplaceBounded}) in the sense of viscosity solutions and in sense of distributions is stable with regards to respectively local uniform limits and $H^\sig$ weak limits.  Therefore by Lemma \ref{lem:AppendixVepCovergeToV}, $v$ is in $H^\sig$ and hence for a.e.$x$, 
\[
\int_{\real^n}\frac{\abs{\del\Gamma(x,y)}}{\abs{y}^{n+\sig}}dy <\infty,
\]
and $v$ satisfies the inequality almost everywhere, and this gives the proposition.

\end{proof}



\bibliography{../refs}
\bibliographystyle{plain}
\end{document}